\documentclass{article}
\usepackage[utf8]{inputenc}
\usepackage
[T1]{fontenc}
\usepackage{microtype}                  
\usepackage{lmodern}                    
\usepackage{longtable}                  
\usepackage{caption}
\usepackage[font=small]{caption}
\usepackage{ltcaption}

\usepackage[inline]{enumitem}           

\usepackage[
  top=20mm, bottom=20mm,
  left=20mm, right=20mm
]{geometry}
\usepackage{setspace}

\usepackage{amsmath, amssymb, amsthm}
\usepackage{mathtools}
\usepackage{bm}                         

\usepackage{booktabs}
\usepackage{siunitx}
\usepackage{algpseudocode}
\usepackage{booktabs}  
\usepackage{siunitx}
\usepackage{multirow}
\usepackage{tabularx}
\usepackage{longtable}
\usepackage{array}
\usepackage{graphicx}
\usepackage{subcaption}
\usepackage{float}
\usepackage{tikz}
\usetikzlibrary{shapes.geometric, arrows.meta, positioning, calc}

\usepackage[table]{xcolor}
\definecolor{headerblue}{RGB}{13, 71, 161}
\definecolor{rowgray}{RGB}{242, 242, 242}

\usepackage[numbers, sort&compress]{natbib}
\usepackage[colorlinks=true,
            linkcolor=headerblue,
            citecolor=headerblue,
            urlcolor=headerblue]{hyperref}

\usepackage{enumitem}

\usepackage{lipsum}                     
\usepackage{comment}                    
\usepackage{algorithm}
\usepackage{algpseudocode}

\algrenewcommand\algorithmicrequire{\textbf{Input:}}
\algrenewcommand\algorithmicensure{\textbf{Output:}}
\theoremstyle{plain}
\newtheorem{theorem}{Theorem}[section]
\newtheorem{lemma}[theorem]{Lemma}
\newtheorem{corollary}[theorem]{Corollary}
\newtheorem{proposition}[theorem]{Proposition}
\newtheorem{conjecture}[theorem]{Conjecture}
\newtheorem{definition}[theorem]{Definition}
\newtheorem{example}[theorem]{Example}

\theoremstyle{remark}
\newtheorem{remark}[theorem]{Remark}

\title{S-LCG: Structured Linear Congruential Generator-Based Deterministic Algorithm for Search and Optimization}
\author{Ahmed Qasim Mohammed, Haider Banka\footnote{Ahmed Qasim Mohammed is with the Department of Computer Science and Engineering, IIT Dhanbad, India, email:22dr0278@iitism.ac.in. Haider Banka is with the Department of Computer Science and Engineering, Indian Institute of Technology Dhanbad, India, email: haider@iitism.ac.in. Anamika Singh is with the Department of Mathematics and Computing, Indian Institute of Technology Dhanbad, India, email: anamikabhu2103@gmail.com.
Corresponding author: Haider Banka.}, and Anamika Singh }
\date{ }
\begin{document}

\maketitle
\begin{abstract}
    This study presents a novel deterministic optimization algorithm based on a special variant of the Linear Congruential Generator (LCG). While conventional algorithms generally operate within the search space, the introduced technique follows a two-level architecture. In particular, an external loop that adaptively balances between exploration and exploitation, while the internal loop evaluates solutions. It is motivated by the intrinsic structure of the generator, the reason behind naming it the Structured Linear Congruential Generator (S-LCG). which enjoys a number of unique characteristics as follows:
1) a memoryless scheme, which ensures non-overlapping sequences based on distinct seeds, thus ensuring no evaluation redundancy;
2) bit splitting representation, which converts LCG states into multi-dimensional points to overcome the Marsaglia lattice effect;  
3) adaptive exploration-exploitation of the generator space, which leads to implicit optimization of the surrogate smooth objective function; and 
4) constant information gathering speed to avoid the problem of premature convergence.
    Extensive testing on 26 benchmark functions across dimensions $d = 2$ to $30$ demonstrates that S-LCG comes within 1\% of the global optimum in 83.3\% of 138 cases (100\% at $ d = 2$, 81.2\% at $ d = 30$) while the nearest competitor GA achieved 75.4\%. Statistical validation shows that S-LCG outperforms eight cutting-edge binary algorithms. Furthermore, its practical value is confirmed by validation on three constrained engineering design problems. In the end, S-LCG offers an optimization framework that is strictly reproducible and requires only one sensitive parameter to be tuned.
\end{abstract}
\textbf{Keywords:} deterministic optimization, linear congruential generator, meta-heuristic, disjoint orbits, zero redundancy, surrogate function, bit-splitting encoding, constrained optimization

\section{Introduction}\label{sec:intro}

    Multimodal continuous functions, constrained engineering design problems, and NP-hard combinatorial problems all require optimization due to the absence of methods that can achieve provable optimality within a reasonable timeframe. This gap enabled meta-heuristic algorithms to thrive as versatile algorithmic frameworks that provide satisfactory high-quality solutions to complex problems~\cite{SorensenGlover2013}.  Lourenço et al.~\cite{Lourenco2003} argue that the best meta-heuristic is one that can be used "without any problem-dependent knowledge". In the context of defining optimization frameworks, two criteria must be met: generalization across a broad spectrum of problems and a minimal number of parameters to configure. Even after decades of development, three fundamental problems remain common to all popular meta-heuristic methods, regardless of whether they are inspired by evolutionary computation, swarm intelligence, or other natural phenomena.
    
    First, \textbf{non-reproducibility} which stems from reliance on the seed that underpins the mechanism for randomness, as true randomness is prohibitively costly. This results in dependence on pseudo-random number generators (PRNG) to furnish the necessary stochastic elements, because the search path is highly sensitive to the initial conditions, such as mutation, crossover, and velocity. This necessitates running algorithms 30 times or more and reporting statistical aggregates ~\cite{Derrac2011,BartzBeielstein2010}, thereby elevating computational costs and complicating fair comparisons with other options. It is noteworthy that reproducing occurs using fixed seed ~\cite{Knuth1997,Luke2013}, although this may restrict the capacity to identify acceptable solutions, as this class of algorithms is designed with stochastic requirements. Second limitation that the algorithms are subject to \textbf{convergence-induced redundancy}, as they generate subsequent candidate solutions from existing ones, which may result in the repetition of previously evaluated solutions. This phenomenon occurs due to the clustering of members in similar regions of the search space, thereby constraining the exploration of new areas. The existing paradigm cannot ensure the absence of redundant samples, and even if it could, it would require massive computational overhead.
    The last is \textbf{parameter sensitivity}. The vast majority of competitive meta-heuristic algorithms use at least two or three parameters that need to be tuned to their optimum values according to the problem at hand. The need for parameter fine-tuning negates the very idea of problem-independent optimization.

    These limitations naturally lead to the question: \emph{can one design an optimizer that retains the adaptive exploration-exploitation balance that makes meta-heuristics effective while being deterministic, minimally parametrized, and structurally non-redundant by construction rather than by heuristic patch?} Deterministic alternatives currently in use only partially address this question. While fully reproducible, low-discrepancy quasi-random sequences~\cite{Sobol1967,Halton1960} employ a fixed evaluation order without an adaptive mechanism. Non-redundancy is achieved only by deterministic search techniques such as DIRECT and tabu search; it comes at a computational overhead. Ultimately, no existing framework successfully eliminates all three limitations.

    During the investigation of this question, intriguing findings emerged by repurposing the algebraic cycle structure of the Linear Congruential Generator (LCG), which is a significant pseudorandom number generator (PRNG). LCG was introduced by Lehmer et al.~\cite{Lehmer1951}, and later thoroughly examined by Knuth et al.~\cite{Knuth1997}. It is characterized by recurrence relation. The recurrence relation has multiple variations, which give it a role of backbone for many PRNGs, such as RANDU introduced by IBM~\cite{Entacher2001Collection}. Owing to its simplicity and fast generating speed, it has been extensively utilized in computational science as a source of randomness. The variations of LCG parameters were the inspiration for the Structured Linear Congruential Generator (S-LCG), a novel deterministic optimization algorithm designed to address the mentioned limitations, drawing from the relationships governing the arrangement of distribution points in d-dimensional space.

    S-LCG's path is completely determined by its parameter settings because it never uses a random draw, but it leverages the implicit random generation with a modular effect. There is no initial seed to initialize or record. Two executions with the same settings will always follow the same search sequence, independent of the underlying hardware or computational environment. This is a stronger guarantee than seed-fixing because the algorithm is not only reproducible but also seed-independent by design. A single run fully describes how it works. Together, these design choices provide a set of algorithmic guarantees that are rarely found in the optimization literature. The primary contribution presented in this paper is to present a fully deterministic, seed-independent, two-level optimization framework that uses adaptive step-size control in the generator space instead of the solution space. Transforming LCG states into high-dimensional points via bit splitting yields an architecture that addresses the structural limitations of the Marsaglia lattice. In addition, the architecture provides an algebraic zero-redundancy property without additional memory overhead. Guaranteeing zero redundancy keeps the information acquisition rate steady over visited generators. As support to the theoretical contributions a 26 classical benchmark functions~\cite{Yao1999,Yang2010TestProblems} in (eight dimensions) and another well-known constrained engineering design problems, robustly supported by Wilcoxon signed-rank and Friedman statistical analyses.

    The rest of the paper is arranged as follows. Section~\ref{sec:related} reviews related work on stochastic meta-heuristics and LCG-based methods.The mathematical foundation with findings introduced in Section~\ref{sec:math}. Section~\ref{sec:theory} analyses the surrogate function with an empirical example. Section~\ref{sec:algo} presents the algorithm. Section~\ref {sec:experiments} presents the experimental setup and results. Lastly, section~\ref{sec:conclusion} concludes the paper and discusses future directions.
    
\section{Related Work}\label{sec:related}

     In literature, meta-heuristic algorithms are primarily divided into two main subdivisions. One of them is the \textbf{trajectory-based methods}, which focus on perturbing a single solution at each iteration and then applying a selection mechanism that determines whether to adopt a newly discovered solution or continue manipulating the current one until certain criteria are met. This iterative process has led to the development of well-known algorithms like simulated annealing~\cite{Kirkpatrick1983, Cerny1985}, tabu search~\cite{Glover1986}, variable neighborhood search~\cite{MladenovicHansen1997}, and iterated local search~\cite{Lourenco2003}. In contrast, another group of algorithms is more sophisticated, as they exploit the implicit parallelism of the currently evaluated multiple solutions to determine the search direction. Which known as \textbf{population-based algorithms}, that inspired by natural phenomena or a natural relationships known under the umbellar of evolutionary strategies most famous genetic algorithms ~\cite{Holland1975, Goldberg1989, Mitchell1996}, differential evolution~\cite{StornPrice1997, PriceStornLampinen2005}, and evolutionary programming~\cite{Fogel1966,Fogel1995}, on the other hand the algorithms inspired by swarm intelligence techniques such as particle swarm optimization~\cite{KennedyEberhart1995,Poli2007} and ant colony optimization~\cite{Dorigo1996,DorigoGambardella1997,DorigoStutzle2004}. Additionally, it also encompasses methodologies that symbolize physical or chemical interactions, such as central force optimization~\cite{Formato2007, Formato2009OPSEARCH}, gravitational search~\cite{Rashedi2009}, and chemical reaction optimization~\cite{LamLi2010, Lam2012}. While in the last decade nature has become the new inspiration, resulting in many algorithms, including grey wolf optimizer~\cite{GWO}, whale optimization~\cite{WOA}, ant lion optimizer~\cite{ALO}, spotted hyena optimizer~\cite{SHO}, grasshopper optimization~\cite{GOA}, Harris hawks optimization~\cite{HHO}, slap swarm algorithm~\cite{SSA}, mountain gazelle optimizer~\cite{MGO}, and Arctic puffin optimization~\cite{APO24}. However, even though many algorithms reported in the literature have proven their ability to tackle complex problems, they do not offer a definitive solution to the three limitations discussed in the introduction. In addition to that, their assertion that the algorithm should be developed is reinforced by the \textbf{No Free Lunch theorems}~\cite{WolpertMacready1997}, which confirm no single meta-heuristic can excel across all function classes. This motivates more investigation into fundamentally distinct search mechanisms, rather than merely variations in the existing nature-inspired paradigm. 
 
    One of the primary concerns identified with the present algorithms is \textbf{convergence-induced redundancy}, which leads to reevaluations of the function due to clustering of existing solutions in previously sampled regions. Techniques such as \textbf{diversity preservation} and \textbf{adaptive parameter mechanisms} are used to mitigate the issue, but they cannot completely eliminate it, as neither addresses the fundamental reality that population-based search lacks a mechanism to ensure non-overlap of evaluated regions.

    LCG, which is one of the main components of the proposed algorithm, has been typically used as a stochastic basis for providing random samples for initialization, selection, mutation, and velocity updates in the existing algorithms. Maucher et al.~\cite{Maucher2011} investigated the sensitivity of various meta-heuristic algorithms with respect to different LCGs. He primarily focused on LCGs with long cycles, as they ensured better-quality randomness. He found that LCGs with a short period and poor spectral properties tend to degrade some algorithms.

    Conversely, meta-heuristic algorithms have also been used to identify optimal parameters for LCGs. Park and Miller~\cite{ParkMiller1988} introduced the "minimal standard" LCG as a benchmark for acceptable statistical quality. The objective in both cases, whether employing the LCG as a randomness source or refining it to improve randomness, is to enhance its effectiveness in generating randomness; its role remains that of a passive tool, fundamentally interchangeable with any other pseudo-random generator of comparable quality.

    One of the major setbacks one faces with respect to LCGs is when they are used in high dimensions. Marsaglia in his paper~\cite{Marsaglia1968} established that $d$-tuples generated from successive LCG outputs occupy at most $\lfloor(d!\cdot m)^{1/d}\rfloor$ parallel hyperplanes in $d$-dimensional space, with coverage gaps between hyperplanes being inaccessible irrespective of seed selection or the quantity of outputs produced. Knuth et al. conducted a comprehensive examination of this lattice structure, whereas the spectral test~\cite{LEcuyer1999} provided a conventional assessment of its severity. However, neither could overcome this limitation.

    Deterministic algorithms have been another direction of research due to their importance in ensuring reproducibility. In their work, Sobol~\cite{Sobol1967}, Halton~\cite{Halton1960}, and Hammersley proved that low-discrepancy sequences fill the unit hypercube more evenly than pseudo-random samples, achieving a discrepancy of order $O((\log N)^d / N)$ in $d$ dimensions~\cite{Niederreiter1992}. Low discrepancy sequences have primarily been utilized to initialize population-based methods~\cite{Maaranen2004} or as the foundation for direct search sampling, as they inherently provide reproducibility and uniform coverage, addressing two of the three criteria. However, they still lack an adaptive mechanism, as the evaluation sequence is predetermined and does not adjust based on information acquired during the search, thereby failing to leverage the exploration-exploitation trade-off fundamental to contemporary meta-heuristics.

    Consequently, due to the absence of adaptive mechanisms, an alternative approach was developed that emphasizes step-size control. Specifically, algorithms like evolution strategies use adaptive mechanisms, such as a 1/5th success rule; similarly, CMA-ES \cite{CMA-ES} uses covariance matrix adaptation. Another algorithm that deploys adaptive variable neighborhood search~\cite{MladenovicHansen1997}, which employs a similar combinatorial approach by systematically altering the neighborhood structure. However, the adaptive step dictates the solution space, generating the subsequent candidate and, in turn, revisiting the same region.

    Based on the optimization methodology developed so far, two main issues have been identified from the literature review. First, the literature considers all three principal conditions separately. No singular methodology simultaneously satisfies all three. Quasi-random algorithms are deterministic and uniformly encompass the state space but lack adaptability. Stochastic meta-heuristics are adaptable yet neither deterministic nor structurally non-redundant; conversely, deterministic adaptive algorithms like DIRECT and tabu search are both deterministic and adaptable, attaining non-redundancy through partitioning or memory rather than inherently. Secondly, although Linear Congruential Generators have been extensively analyzed for randomness and optimized to enhance statistical properties, their algebraic characteristics, specifically how seed variation partitions the state space into distinct cycles, have not been utilized as an independent optimization algorithm. The next section examines and applies the algebraic properties of the problem.

\section{Mathematical Foundations}\label{sec:math}
    This section establishes all the mathematical results regarding $S$-$LCG$, which are required for developing the proposed meta-heuristic algorithm. The section discusses the theorems and proofs that support the algorithm's claims for handling the limitations of existing metaheuristic algorithms.
    The standard notations used throughout this paper are defined as follows:
\begingroup
\small 
\begin{description}[
style=multiline, leftmargin=2cm, font=\normalfont,
itemsep=0pt,
parsep=0pt,
topsep=4pt
]
    \item[$S$] Binary string.
    \item[$n$] Total number of bits used in the representation of the encoded number.
    \item[$d$] Dimensionality of the problem, i.e., the number of decision variables.
    \item[$b$] Number of bits assigned to each variable under the bit-splitting scheme.
    \item[$m$] Modulus defining the size of the state space.
    \item[$a,\; c$] Multiplier and Incremental values in LCG recurrence
    \item[$\alpha_0$] Initial generator.
    \item[$\alpha_k$] Intermediate generator.
    \item[$\alpha_{\max}$] last generator in space of $2^n$ (see Conjecture~\ref{conj:alphamax}).
    \item[$\delta$] Step size used during exploitation; fixed to $2$ in this work.
    \item[$\Delta$] Step size controlling exploration, and the primary sensitive parameter in the method.
    \item[$S_{\max}$] Maximum allowed stagnation count before triggering exploration (typically set to 5000).
    \item[$E_{\max}$] Budget for exploitation steps per improving generator (default value: 60).
    \item[$G$] Set of generators, defined as $G = \{\alpha_0, \alpha_1, \ldots, \alpha_{\max}\}$.
    \item[$g(\alpha)$] Surrogate objective defined as $g(\alpha) = \min_{i} f(\mathbf{x}_i(\alpha))$.
    \item[$N_{\text{pop}}$] Population size in a population-based meta-heuristic algorithm.
    \item[$T_{\max}$] Number of maximum iterations in a population-based meta-heuristic algorithm.
\end{description}
\endgroup

 \subsection{LCG Background}\label{sec:lcg}
 
    \begin{definition}[Linear Congruential Generator]\label{def:lcg}
        A Linear Congruential Generator (LCG) is a recursive modular function defined over residue classes modulo $m$ ($\mathbb{Z}_m$) as follows:
    \begin{equation}\label{eq:lcg}
        x_{n+1} = (a \cdot x_n + c) \bmod m, \quad n \geq 0,
    \end{equation}
    where $m$ is a positive integer and is called the modulus, $a$ is the multiplier ($0 < a < m$), $c$ is the increment ($0 \leq c < m$), and $x_0$ is the initial seed.
    \end{definition}

    In this study, the LCG parameters are restricted to $m = 2^n+1$, where $a = 2$, and $c = 1$ are used for all the experiments and referred to as the Structured Linear Congruential Generator or S-LCG in short. These parameters have been carefully chosen in order for the LCG to satisfy certain properties that enable the design of the desired algorithm (Algorithm \ref{sec:algo}). The traditional LCG's parameters were so chosen that it optimized the period length, ideally reaching a full period in accordance with the Hull-Dobell conditions~\cite{HullDobell1962}, whereas S-LCG necessitates the contrary. It requires a configuration that violates the full-period conditions, leading to a state space that disintegrates into numerous brief, disconnected cycles. Every cycle turns into a separate search path that doesn't overlap with any other path. The algebraic guarantee of disjointness (Theorem~\ref{lem:disjoint}) turns what is usually seen as a flaw in the generator into the main part of the optimization algorithm. For the sake of completeness, S-LCG is formally defined as follows:

    \begin{definition}[Structured Linear Congruential Generator (S-LCG)]\label{def:slcg}
        For $n\in \mathbb{N}$, a \emph{Structured Linear Congruential Generator} is defined as a linear congruential generator with the following recurrence relation:
        \begin{equation}\label{eq:slcg}
            x_{k+1} = (2 \cdot x_k + 1) \pmod{2^n+1}, \quad k \geq 0,
        \end{equation}
    \end{definition}

    The LCG generates a series of integers from $\{0, 1, \ldots, m-1\}$ based on an initial seed $x_0$. Because $m$ is finite, the sequence will eventually repeat, creating a \emph{cycle} with a length known as the \emph{period}.
    
    \begin{definition}[Period of an LCG]\label{def:period}
        The \emph{period} of the sequence $\{x_n\}$ generated by Equation~\eqref{eq:lcg} is the smallest positive integer $p$ such that $x_{n+p} \equiv x_n \pmod{m}$ for all sufficiently large $n$.
    \end{definition}
 
    The Hull-Dobell theorem~\cite{HullDobell1962} establishes that an LCG achieves full period (period $= m$) if and only if: (i)~$\gcd(c, m) = 1$; (ii)~$a \equiv 1 \pmod{p}$ for every prime factor $p$ of $m$; and (iii)~if $4 \mid m$, then $a \equiv 1 \pmod{4}$. One can easily verify that the S-LCG violates each of the three Hull-Dobell full-period conditions. In fact, the next theorem (Theorem \ref{thm:cycle_length}) shows that the period of S-LCG is at most $2n$.

    \begin{theorem}\label{thm:cycle_length}
        The period of the sequence generated by any generator is at most $2n$.
    \end{theorem}
    
    \begin{proof}
    The recurrence relation corresponding to the S-LCG general equation (\ref{eq:slcg}) and the closed-form solution for the \( k \)-th term is:
    \begin{equation} \label{eq1}
        x_k = 2^k x_0 + (2^k - 1) \mod (2^n + 1)
    \end{equation}
    Equation \ref{eq1} can be easily verified by induction. The base case \( k=0 \) is trivial. Assuming it holds for \( k \), then:
    \begin{align*}
    x_{k+1} &= 2x_k + 1 \\
            &= 2\left[2^k x_0 + (2^k - 1)\right] + 1 \\
            &= 2^{k+1}x_0 + 2^{k+1} - 2 + 1 \\
            &= 2^{k+1}x_0 + (2^{k+1} - 1),
    \end{align*}
    which confirms the inductive step.
    Let \( p \) be the period for a given seed \( x_0 \), so \( x_p = x_0 \). Applying \ref{eq1}, one gets:
    \[
        2^p x_0 + (2^p - 1) \equiv x_0 \pmod{2^n + 1}.
    \]
    Rearranging terms gives the fundamental period condition:   
    \begin{equation} \label{eq1b}
         (2^p - 1)(x_0 + 1) \equiv 0 \pmod{2^n + 1}
    \end{equation}
    Let \( d = \gcd(x_0 + 1, 2^n + 1) \). Equation \ref{eq1b}, it follows that \( \frac{2^n + 1}{d} \mid (2^p - 1) \), which is equivalent to:    
    \begin{equation} \label{eq2}
        2^p \equiv 1 \pmod{\frac{2^n + 1}{d}}
    \end{equation}
    Let \( \text{ord}_N(2) \) denote the multiplicative order\footnote{The \emph{multiplicative order} $\operatorname{ord}_N(a)$ is the smallest positive integer $k$ such that $a^k \equiv 1 \pmod{N}$, where $\gcd(a, N) = 1$.} of 2 modulo \( N \). Equation \ref{eq2} implies that \( p \) is a multiple of \( \text{ord}_{\frac{2^n+1}{d}}(2) \).
    It can be easily be seen that \( \text{ord}_{2^n+1}(2) = 2n \), since \( 2^{2n} = (2^n)^2 \equiv (-1)^2 \equiv 1 \pmod{2^n+1} \), and no smaller positive exponent equals 1. The order of 2 modulo any divisor of \( 2^n+1 \) must divide \( 2n \). Therefore, \( \text{ord}_{\frac{2^n+1}{d}}(2) \mid 2n \), and hence the period \( p \mid 2n \). This proves that the period for \emph{any} seed is \emph{at most} \( 2n \).
    \end{proof}
    
    \begin{remark}
        The upper bound \( 2n \) is achievable. From the above, the period equals \( 2n \) if \( \text{ord}_{\frac{2^n+1}{d}}(2) = 2n \). This occurs when \( d = 1 \), i.e., when \( \gcd(x_0 + 1, 2^n + 1) = 1 \). In this case, (2) implies \( 2^p \equiv 1 \pmod{2^n + 1} \), forcing \( p = 2n \).
    \end{remark}

    \noindent\textbf{Note:}
        A simple seed satisfying this condition is \( x_0 = 0 \), since \( \gcd(1, 2^n + 1) = 1 \). Therefore, the S-LCG with seed \( x_0 = 0 \) achieves the maximum period \( 2n \).

\subsection{Properties of \texorpdfstring{S-LCG}{S-LCG}}\label{sec:properties slcg}
    This section discusses some other significant characteristics of S-LCG. The ability of S-LCG to cover discrete points in the search space has been demonstrated. This is demonstrated in two stages: first, it is established that S-LCG partitions $\mathbb{Z}_m$ into a set of disjoint cycles. (Lemma~\ref{lem:disjoint}); second, it is illustrated that designating the minimal element of each cycle as its \textit{generator} produces a unique representative system (Theorem~\ref{thm:unique_rep}).
    

    \begin{lemma}\label{lem:disjoint}
        The map $f$ associated with S-LCG partitions the state space $\mathbb{Z}_m$ into a collection of disjoint cycles.
    \end{lemma}

    \begin{proof}
        Let $n \geq 1$, $m = 2^n + 1$, and let $f : \mathbb{Z}_m \to \mathbb{Z}_m$ be the linear congruential map defined by
    \[
        f(x) \;=\; 2x+1 \pmod{m}.
    \]
    Then, $f$ is a bijection. Since $m$ is odd, $\gcd(2, m) = 1$, so $2$ is invertible modulo $m$, and the map $f^{-1} : \mathbb{Z}_m \to \mathbb{Z}_m$ 
    given by
    \[
        f^{-1}(y) \;=\; 2^{-1}(y - 1) \pmod{m}
    \]
    is well-defined and satisfies $f^{-1} \circ f = \mathrm{id}_{\mathbb{Z}_m}$. Hence $f$ is a bijection and therefore induces a permutation on the finite 
    set $\mathbb{Z}_m$. This permutation structure inherently defines an equivalence relation where elements are equivalent if they belong to the same orbit (or cycle) under iteration by \( f \). This, in turn, partitions the search space \( \mathbb{Z}_m \) into a set of disjoint cycles.    
    \end{proof}

    Every element in the cycle produced by $S$-$LCG$ can generate the entire cycle. In order to provide unique coverage, the following theorem establishes a representative system for each cycle that allows the selection of only one element per cycle. This representative is referred to as the generator of the cycle, which prevents repetitions.
    
    \begin{theorem}[Unique Cycle Coverage]\label{thm:unique_rep}
    Let $\{C_i\}$ denote the collection of disjoint cycles generated by $S$-$LCG$, that partition $\mathbb{Z}_m$. Define the 
    \emph{generator} of each cycle $C_i$ as follows:
    \[
        \alpha_i \;=\; \min\, C_i.
    \]
    Then the set of generators $G = \{\alpha_i\}$ constitutes a system of distinct representatives such that an element of $\mathbb{Z}_m$ is associated with precisely one cycle and is accessible from exactly one generator of $S$-$LCG$.
    Moreover, the sequential procedure of scanning $\mathbb{Z}_m$ in increasing order from $0$ and admitting each element not yet visited as a new generator, 
    then marking its entire cycle as visited, is equivalent to enumerating $G$ and guarantees that no cycle is sampled more than once.
    \end{theorem}

    \begin{proof}

    Since $\{C_i\}$ is a partition of $\mathbb{Z}_m$, every element $x \in \mathbb{Z}_m$ belongs to exactly one cycle $C_i$. By the definition of a cycle under $f$, every element of $C_i$ is reachable from $\alpha_i$ by finitely many iterations of $f$, and 
    from no other generator.

    \textit{Equivalence of the sequential procedure:} Let the scanning procedure be executed as described. An element $x \in \mathbb{Z}_m$ is admitted as a generator if and only if no element of its cycle $C_i$ has been previously admitted. Among all elements of $C_i$, the first to be encountered in increasing order is precisely $\alpha_i = \min\, C_i$. Upon admission, all elements of $C_i$ are marked visited, so no subsequent element of $C_i$ can be admitted, and no element of any other cycle is affected. Consequently, exactly one generator is admitted per cycle, and the admitted set equals $G$. This ensures that each cycle contributes exactly one seed to the enumeration, and no cycle is repeated.
    \end{proof} 
    The generator-selection procedure established in Theorem~\ref{thm:unique_rep} is used to devise the Algorithm~\ref{alg:generator_check} that helps to determine whether an even number in $\mathbb{Z}_m$ is a generator or not.

    \begin{figure}[H]
    \centering
    
    \begin{minipage}[t]{0.4\textwidth}
        \vspace{0pt} 
        \footnotesize 
        
        \hrule height 0.8pt
        \vspace{1pt} 
        \captionof{algorithm}{Check if $\alpha_k$ is a generator}
        \label{alg:generator_check}
        \vspace{-10pt} 
        \hrule height 0.4pt
        
        \begin{algorithmic}[1]
        \Require $\alpha_k$, $n$
        \Ensure $1$ if generator, $-1$ otherwise

        \State $MASK \gets (1 \ll n) - 1$
        \State $complement \gets MASK - \alpha_k$
        \State $state \gets \alpha_k$
        \State $LookUpTable \gets [1,0]$

        \For{$i \gets 0$ to $n$}
            \If{$\alpha_k \le state \le complement$}
                \Return $-1$
            \EndIf
            \State $state \gets ((state \ll 1) \mathbin{\&} MASK) \mid (LookUpTable[(state \gg n-1) \mathbin{\&} 1])$
        \EndFor

        \Return $1$
        \end{algorithmic}
        
        \hrule height 0.8pt
        
    \end{minipage}\hfill
    \begin{minipage}[t]{0.56\textwidth}
        \vspace{0pt} 
        \centering
        \footnotesize 
        \setlength{\tabcolsep}{1.5pt} 
        \renewcommand{\arraystretch}{0.9} 
        
        \begin{tabular}{c|ccccccccccccc}
        \toprule
        \textbf{Gen.} & \textbf{S1} & \textbf{S2} & \textbf{S3} & \textbf{S4} & \textbf{S5} & \textbf{S6} & \textbf{S7} & \textbf{S8} & \textbf{S9} & \textbf{S10} & \textbf{S11} & \textbf{S12} & \textbf{S13} \\
        \midrule
        0  & 1  & 3  & 7   & 15  & 31 & 63 & 127 & 126 & 124 & 120 & 112 & 96  & 64 \\ 
        2  & 5  & 11 & 23  & 47  & 95 & 62 & 125 & 122 & 116 & 104 & 80  & 32  & 65 \\ 
        4  & 9  & 19 & 39  & 79  & 30 & 61 & 123 & 118 & 108 & 88  & 48  & 97  & 66 \\ 
        6  & 13 & 27 & 55  & 111 & 94 & 60 & 121 & 114 & 100 & 72  & 16  & 33  & 67 \\ 
        8  & 17 & 35 & 71  & 14  & 29 & 59 & 119 & 110 & 92  & 56  & 113 & 98  & 68 \\ 
        10 & 21 & 43 & 87  & 46  & 93 & 58 & 117 & 106 & 84  & 40  & 81  & 34  & 69 \\ 
        12 & 25 & 51 & 103 & 78  & 28 & 57 & 115 & 102 & 76  & 24  & 49  & 99  & 70 \\ 
        18 & 37 & 75 & 22  & 45  & 91 & 54 & 109 & 90  & 52  & 105 & 82  & 36  & 73 \\ 
        20 & 41 & 83 & 38  & 77  & 26 & 53 & 107 & 86  & 44  & 89  & 50  & 101 & 74 \\ 
        42 & 85 &    &     &     &    &    &     &     &     &     &     &     &    \\ 
        \bottomrule
        \end{tabular}
        
        \captionof{table}{State transitions for different generators for $n = 7$.}
        \label{tab:generators_states}
    \end{minipage}
\end{figure}

    Algorithm $1$ check is $\alpha_k$ is a generator, mathematically proven (Theorem \ref{thm:unique_rep}), memory-free mechanism that guarantees every valid seed produces a sequence completely disjoint from every other valid seed's sequence. No storage, O(n) cost.

    \begin{corollary}
    \label{cor: zero_redundancy}
        After testing K valid seeds, the total number of unique d-dimensional points evaluated is exactly $2n \cdot K$.
    \end{corollary}
    
    Having established a unique generator selection, a structural property arising directly from S-LCG is now identified.
    Proposition~\ref{prop:even_generators} shows that every generator produced by S-LCG is necessarily even, a consequence not of explicit construction, but of the interplay between the multiplier, increment, and odd modulus.

    \begin{proposition}\label{prop:even_generators}
        Every generator of S-LCG is even.
    \end{proposition}
    \begin{proof}
        (By contradiction) assume that $C$ is a cycle generated by the generator $\alpha$ and $\alpha$ is odd. As the S-LCG is a bijective map, there exists a pre-image of $\alpha$ given by:
        \[
            p = 2^{-1}(\alpha -1) \pmod{2^n + 1}
        \]
        Since $\alpha$ is odd therefore $\alpha - 1$ is even and $\frac{\alpha - 1}{2} \in \{0, 1, \ldots, 2^n -1\}$. 
        Observe that $p = 2^{-1}(\alpha -1) \equiv \frac{\alpha - 1}{2} \pmod{2^n + 1}$. But, $\frac{\alpha - 1}{2} < \alpha$ leading to the contradiction that a generator is the minimal element in the cycle. Hence, a generator is always even.
    \end{proof}

    With the even-generator property established, the internal structure of each cycle is now examined. In the next theorem(Theorem~\ref{thm:cycle_length}), it is shown that the length of every cycle is governed by $n$ and is at most $2n$.

    \subsection{The Last Generator \texorpdfstring{($\alpha_{max}^{(n)}$)}{\alpha_{max}^{(n)}} }\label{sec: alpha_max}

    Section \ref{sec:properties slcg} establish that the set of generators $G = \{\alpha_1, \alpha_2, ..., \alpha_k\}$ is well-defined and finite. The largest element of G, denoted $\alpha_{max}^{(n)}$, determines the termination condition of the algorithm. Through extensive numerical experimentation across [range of parameters tested], it is observed that $\alpha_{\max}^{(n)}$ satisfies the following relationship:

    \begin{conjecture}[Maximum Generator]\label{conj:alphamax}
        For any odd $n \in \mathbb{N}$, the maximum valid generator $\alpha_{\max}^{(n)}$ for S-LCG admits both a recursive and a closed-form description.\\\
        \textbf{Closed form:}\\
        \begin{equation}
        \alpha_{\max}^{(n)} = \frac{4}{3}\Big(4^{(\frac{n-3}{2})} - 1\Big).
        \end{equation}
        \textbf{Recursive form:}\\
        \begin{equation}
        \alpha_{max}^{(2k+1)} = 4 \times \alpha_{max}^{(2k-1)} + 4, \quad\quad 2 < k < 30 , \text{ and } \alpha_{max}^{(3)} = 0.  
        \end{equation}
    \end{conjecture}

    \noindent\textit{Numerical evidence.} Conjecture~\ref{conj:alphamax} has been verified computationally for all odd $n \leq 53$ listed in Table~\ref{alpha_max_odd}. No counterexample has been found across the above tested cases. A formal proof remains an open problem.

\begin{table}[H]
    \centering
    \footnotesize 
    \setlength{\tabcolsep}{4pt} 
    \renewcommand{\arraystretch}{0.9} 
    
    \begin{tabular}{
        S[table-format=2.0]
        S[table-format=8.0]
        S[table-format=9.0]
        S[table-format=2.3]
        @{\hspace{3em}} 
        S[table-format=2.0]
        S[table-format=16.0]
        S[table-format=16.0]
        S[table-format=2.3]
    }
    
    \toprule
    {$n$} & {$\alpha_{\max}^{n}$} & {$2^n$} & {Ratio (\%)} & 
    {$n$} & {$\alpha_{\max}^{n}$} & {$2^n$} & {Ratio (\%)} \\
    \midrule
    3  & 0        & 8         & 0.000  & 29 & 89478484         & 536870912        & 16.667 \\
    5  & 4        & 32        & 12.500 & 31 & 357913940        & 2147483648       & 16.667 \\
    7  & 20       & 128       & 15.625 & 33 & 1431655764       & 8589934592       & 16.667 \\
    9  & 84       & 512       & 16.406 & 35 & 5726623060       & 34359738368      & 16.667 \\
    11 & 340      & 2048      & 16.602 & 37 & 22906492244      & 137438953472     & 16.667 \\
    13 & 1364     & 8192      & 16.650 & 39 & 91625968980      & 549755813888     & 16.667 \\
    15 & 5460     & 32768     & 16.663 & 41 & 366503875924     & 2199023255552    & 16.667 \\
    17 & 21844    & 131072    & 16.666 & 43 & 1466015503700    & 8796093022208    & 16.667 \\
    19 & 87380    & 524288    & 16.666 & 45 & 5864062014804    & 35184372088832   & 16.667 \\
    21 & 349524   & 2097152   & 16.667 & 47 & 23456248059220   & 140737488355328  & 16.667 \\
    23 & 1398100  & 8388608   & 16.667 & 49 & 93824992236884   & 562949953421312  & 16.667 \\
    25 & 5592404  & 33554432  & 16.667 & 51 & 375299968947540  & 2251799813685248 & 16.667 \\
    27 & 22369620 & 134217728 & 16.667 & 53 & 1501199875790164 & 9007199254740992 & 16.667 \\
    \bottomrule
    \end{tabular}
    \caption{Exhaustively verified values for $\alpha_{\max}$, including ratio of $\alpha_{max}$ to $2^n$.}
    \label{alpha_max_odd}
\end{table}

\section{Theoretical Analysis}
\label{sec:theory}
This section discusses the foundational theories of the proposed algorithm. First, the technique to mitigate the curse of dimensionality associated with using LCG in high dimensions is described. Next, the concept of reducing the d-dimensional search space to a 1-dimensional surrogate landscape is introduced. Finally, the algorithm maintains a constant information acquisition rate due to the algebraic properties of $S$-$LCG$, is established.

\subsection{Bit-Splitting Encoding}
\label{sec:bit-split}

    Classically, when an LCG is used to generate points in multi-dimensional spaces, it suffers from a structural flaw identified by Marsaglia~\cite{Marsaglia1968}. Marsaglia's theorem states that $d$-dimensional tuples formed by successive LCG outputs fall on a relatively small number of parallel hyperplanes. Specifically, the number of these hyperplanes is bounded above by $\lfloor(d! \cdot m)^{1/d}\rfloor$. This phenomenon creates a systematic coverage gap, making the standard LCG a poor choice for discretely approximating continuous search spaces. In order to bypass Marsaglia's issue, a bit-splitting encoding technique is proposed, which is inspired from Genetic Algorithms. Bit-splitting encoding is formally defined as follows:

    \begin{definition}[Bit-Splitting Encoding]
        Let $[L_j, U_j]$ be the search bound for the $j$-th variable, and let $b_j$ be its allocated bit width, where $\sum_j b_j = n$. Partition the $n$-bit state $S$ of the S-LCG into contiguous segments $S^{(1)}, S^{(2)}, \ldots$ where each S$^{(j)} \in \{0, 1, \ldots, 2^{b_j}-1\}$ is the integer value of the $j$-th segment. Then, the $j$-th coordinate corresponding to a state $S$ of the S-LCG is given by:
    \label{def:bit-split}
        \begin{equation}\label{eq:bitsplit}
            v_{j} \;=\; L_{j} + \frac{S_j}{2^{b_j} - 1} \cdot \bigl(U_{j} - L_{j}\bigr).
        \end{equation}
    \end{definition}
    The following example demonstrates how one state of S-LCG gives rise to a point in 3D:
    \begin{example}
        $S=001|101|101 \Rightarrow (5,5,1),\;
    v_j=-1+\frac{S_j}{7}\cdot 2,\;
    \mathbf{v}=(0.4285,0.4285,-0.7142)$. Because all $d$ coordinates are extracted from a \emph{single} state rather than consecutive states, no linear dependency exists among the dimensions. This establishes a true bijection between the state space and the $d$-dimensional discretized grid. 
    \end{example}
        \begin{figure}[H]
    \centering
    \includegraphics[height=0.20\textheight, keepaspectratio]
    {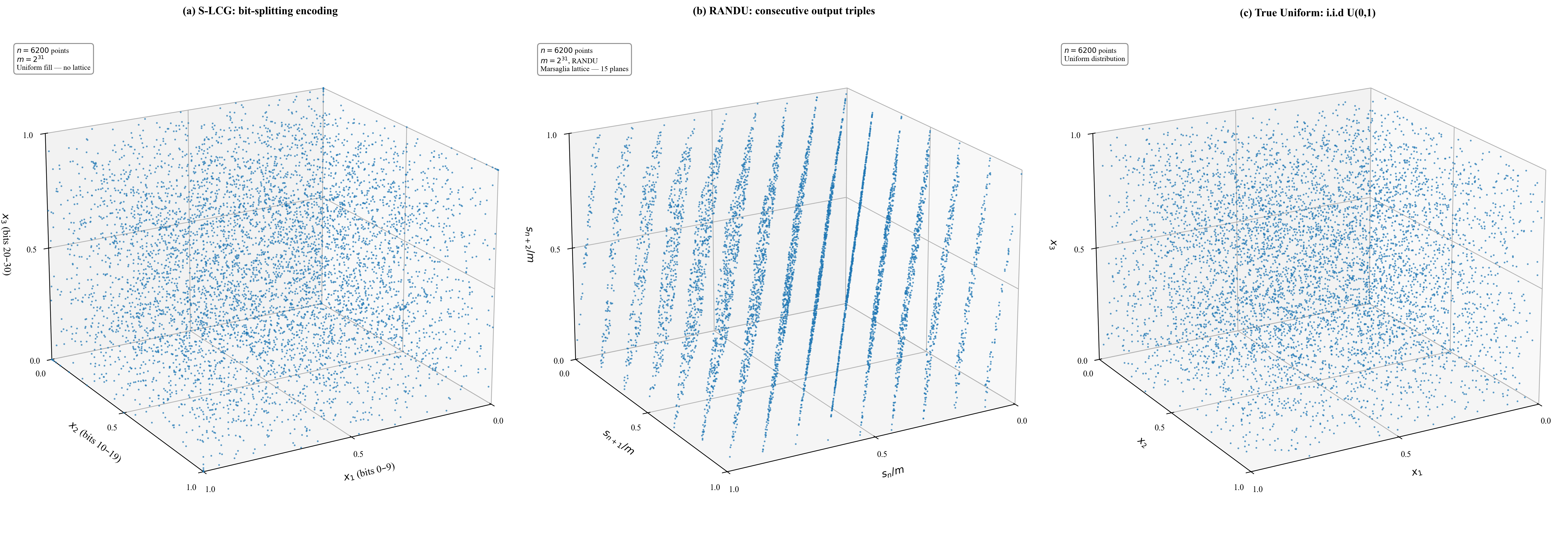}
    \caption{Three-dimensional representations of the spatial distributions produced by S-LCG, RANDU, and Uniform distribution.}
    \label{Marsaglia_Comparison}
    \end{figure}
    
    Table~\ref{tab:uniformity} illustrates the distribution of points produced by S-LCG, Randu, and uniform distributions for the study of uniformity. The outcomes of S-LCG in this study are satisfactory for algorithm establishment, utilizing 6,200 points in the three-dimensional space of $2^{31}$. The small correlation found is an expected behavior of the S-LCG due to the use of step size ($\Delta$) between sampled generators. Figure~\ref{Marsaglia_Comparison} gives a visual representation of the distribution of points for the above-considered case. The figure shows that the points generated by RANDU exhibit parallel hyperplane clustering (Marsaglia's effect), in contrast to those generated by the uniform distribution and S-LCG.

    \begin{table}[H]
    \small
    \centering
    \footnotesize 
    \setlength{\tabcolsep}{4pt} 
    \renewcommand{\arraystretch}{0.9} 
    \begin{tabular}{@{}llccc@{}}
    \toprule 
    Test & Metric & S-LCG & RANDU & True Uniform \\
    \midrule 
    Chi-square ($5^3$ bins) & $p$-value & \textbf{0.298} & 0.117 & 0.109 \\
    KS (worst dimension) & $D$-statistic & \textbf{0.0101} & 0.0110 & 0.0143 \\
    Nearest-neighbor & Mean/Expected & \textbf{0.984} & 0.706 & 1.018 \\ 
    Nearest-neighbor & CV & \textbf{0.412} & 0.543 & 0.372 \\
    Max $|$Correlation$|$ & $|r|$ & 0.029 & 0.008 & 0.028 \\ 
    \bottomrule 
    \end{tabular} 
    \caption{Results of uniformity.}
    \label{tab:uniformity} 
    \end{table}

    \subsection{Surrogate Function Transformation}\label{sec:surrogate}
    
    The core strategy that the proposed algorithm uses in order to solve a continuous, multi-dimensional optimization problem is by converting the multi-dimension objective function into a 1-dimensional, discrete surrogate function. This is achieved via a two-level nested loop architecture. The outer loop iterates over the generator space $G$, and for each generator, an inner loop finds the minimum function value over the cycle produced by the generator. This process implicitly constructs a surrogate function $g(\alpha)$ over the discrete generator space which is defined as follows:
    \begin{definition}[Surrogate Function]
        Let $f:  D \subseteq \mathbb{R}^d  \rightarrow \mathbb{R}$ be an objective function that needs to be minimized, and let $G$ be set of generators. For each generator $\alpha \in G$, let $C(\alpha) = \{x_1, x_2, \dots, x_{2n}\}$ denote the set of candidate solutions produced by cycle of $\alpha$ and normalized to the search domain, then, the surrogate function $g: G \rightarrow \mathbb{R}$ is defined as follows:
        \[
             g(\alpha) = \min_{x \in C(\alpha)} f(x).
        \]
        
    \end{definition}
    Thus, by mapping the exhaustive cycle evaluation into the single operator $g(\alpha)$, the original problem, $\arg\min_{x \in D} f(x)$, is transformed into finding $\arg\min_{\alpha \in G} g(\alpha)$. Hence, by the help of surrogate function, S-LCG is able to handle the curse of dimensionality which the traditional population-based meta-heuristic methods finds difficult to deal with due to exponential increase in the volume of the search space with the increase in dimension. Furthermore, the generator space is drastically bounded to a computationally tractable domain (Section~\ref{sec: alpha_max}). The empirical $\alpha_{max}$ conjecture bounds the generator space to $\alpha \in G \cap [0, \alpha_{max}]$, effectively reducing the global search space to approximately 16\% of the total integer domain.

    \subsubsection{Navigating the surrogate landscape}
    Considering that all the generators in $G$ are even (Proposition~\ref{prop:even_generators}), an increment of $2$ in the outer loop, starting from the generator $0$ until $\alpha_{max}$, takes us through all the generators in $G$. However, the total number of generators in the space is around $\frac{2^n}{2n}$, which is too large for exhaustive traversal for large $n$. This necessitates an intelligent step control strategy within the generator space. But, for the surrogate function $g(\alpha)$ to be optimizable via adaptive step-control, it must exhibit localized topographical structure rather than pure random noise.
    
    For small value of $\delta$, the two generators, $\alpha$ and $\alpha + \delta$, are next to each other in $\mathbb{Z}_m$, hence they produce cycles $C(\alpha)$ and $C(\alpha + \delta)$ that are adjacent subspaces in $\mathbb{R}^d$. This makes point sets in $d$-dimensional space partly correlated because their starting conditions are so similar. This means that $g(\alpha)$ and $g(\alpha + \delta)$ are compared in nearby areas of the search space, which causes local correlation within $g$.
    
    The adaptive step control strategy used by S-LCG alternate between exploitation ($\delta$-step) and exploration ($\Delta$-step). $\delta$-step is applied when $g(\alpha)$ yields an improvement. Improvement in $g(\alpha)$ indicates a descent direction in the surrogate landscape. A small step ($\delta$) allows the algorithm to meticulously exploit this correlated sub-space but when local exploitation stagnates, a large leap ($\Delta$) shatters the bit-level correlation, projecting the search into a distinct, algebraically uncorrelated region of $G$ to escape local optima. The adaptive step control strategy will be elaborated more in Algorithm Section~\ref{sec:algo}.

    To quantify the effect of adaptive step control strategy in surrogate landscape, the F8 function(Schwefel) is considered as it is known for being notoriously deceptive. The assessment is conducted in the space of $\mathbb{Z}_{2^{21}}$ for two dimension. It reveals a highly structured landscape rather than erratic noise as illustrated in Figure~\ref{fig:Surrogate_function_Schwefel}.
    
    \begin{figure}[H]
    \centering
    \includegraphics[width=0.8\linewidth , height=0.25\textheight]{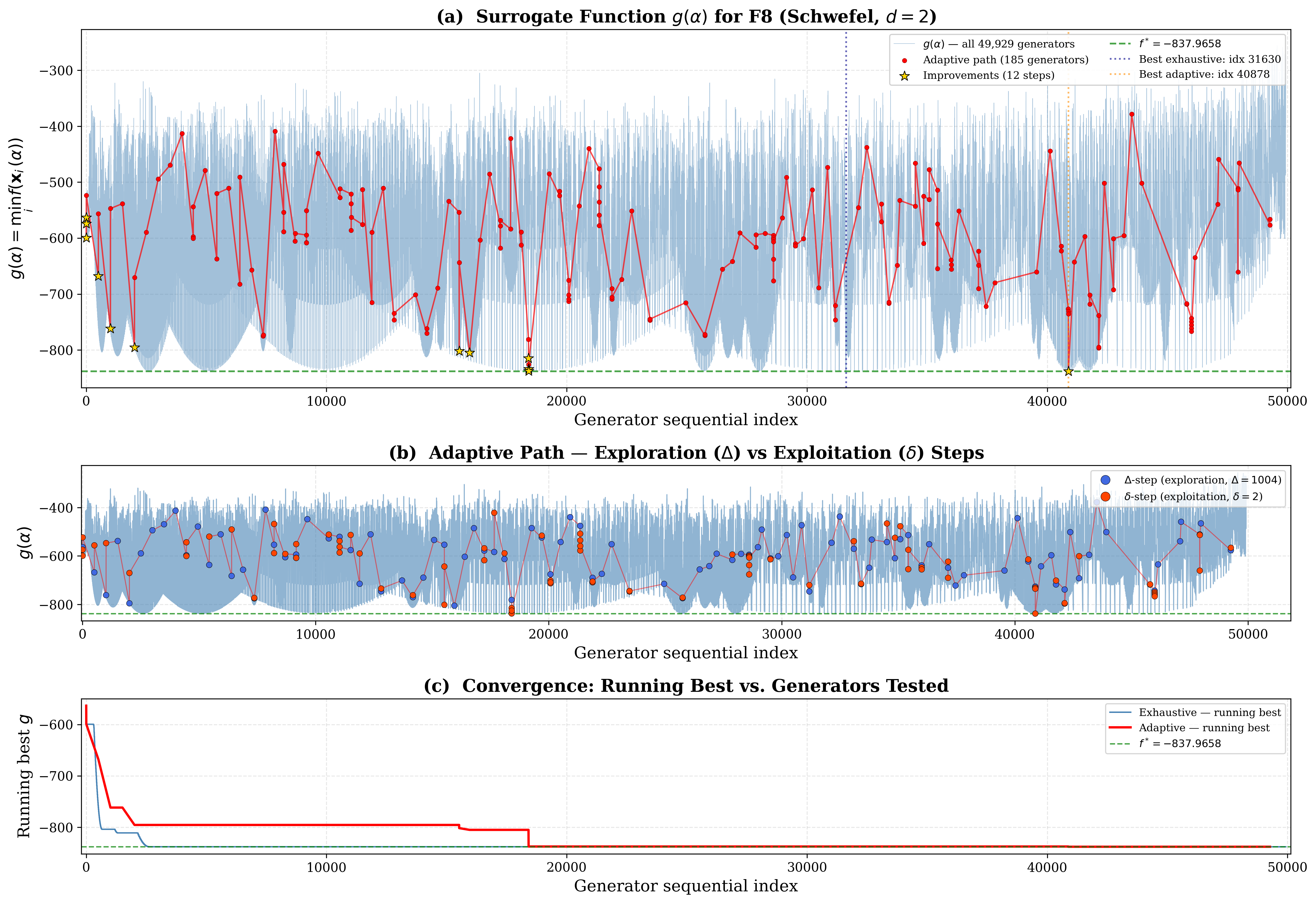}
    \caption{Exhaustive versus adaptive evaluation of the surrogate landscape $g(\alpha)$ on the deceptive Schwefel function.}
    \label{fig:Surrogate_function_Schwefel}
\end{figure}

    The practical efficiency of this smoothing is profound. As detailed in Table~\ref{tab:surrogate}, an adaptive traversal using the $\delta = 2 \text{ and } \Delta = 1004 $ required evaluating only 185 generators (0.37\% of the space) to achieve a solution within 0.016\% of the known global optimum. 
\begin{table}[H]
    \centering
    \begin{tabular}{lrrr}
        \toprule
        \textbf{Strategy} & \textbf{Generators Tested} & \textbf{Best $g(\alpha)$} & \textbf{Gap from $f^{*}$} \\
        \midrule
        Exhaustive & 49{,}929 & $-837.9615$ & 0.0005\% \\
        Adaptive   &      185 & $-837.8338$ & 0.0160\%  \\
        \bottomrule
    \end{tabular}
    \caption{Exhaustive vs.\ adaptive sweep efficiency on F8 (Schwefel, $d=2$).}
    \label{tab:surrogate}
\end{table}
The exploration step $\Delta$ is a large sampling interval, similar to the sampling theorem of Nyquist-Shannon~\cite{Nyquist1928Telegraph, Shannon1949Communication}. The smoothness of $g(\alpha)$ provide wide sampling sufficient to find a productive valleys, which reduce computational cost of exhaustive enumeration.

To confirm existing of correlation between consecutive generator  lag-$k$ autocorrelation used, defined as the Pearson correlation between $g(\alpha_j)$ and $g(\alpha_{j+k})$ on the whole generators space, which confirms strong local smoothness: lag-1 yields $\rho_1 = 0.6494$, decaying gradually to $\rho_{20} = 0.0452$ Figure.~\ref{fig:rho_Schwefel}. This localized correlation provides the strict mathematical justification for the $\delta$-step exploitation mechanism.

\begin{figure}[H]
    \centering
    \includegraphics[width=0.7\linewidth]{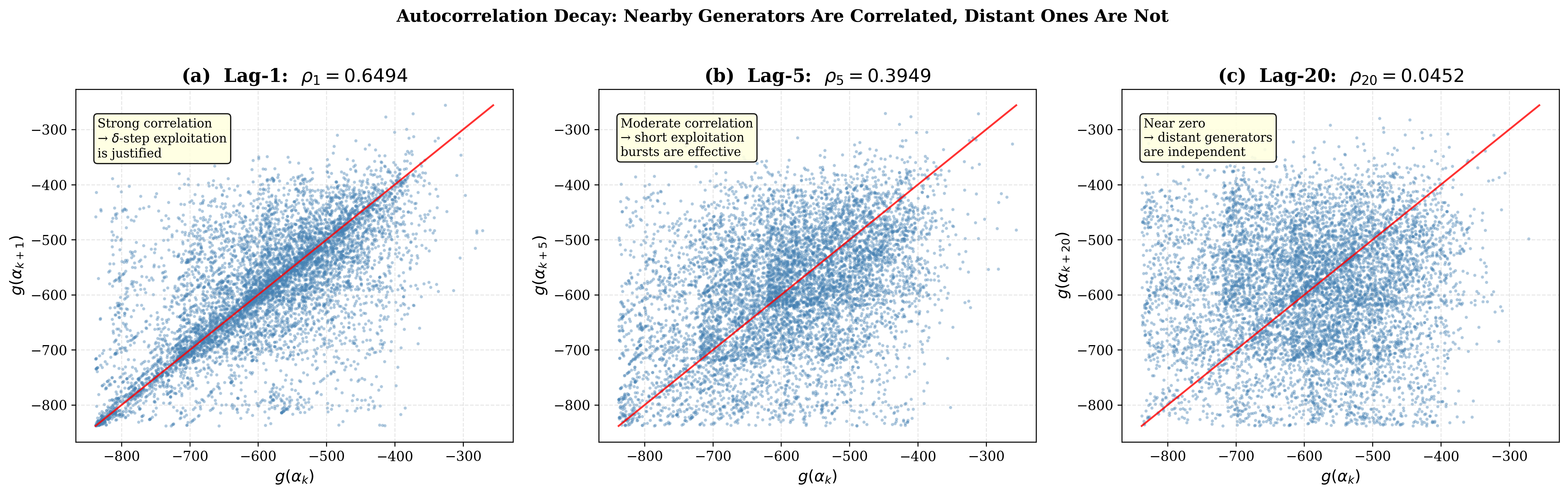}
    \caption{Lag-$k$ autocorrelation in the surrogate landscape, demonstrating localized topological.}
    \label{fig:rho_Schwefel}
\end{figure}

\subsection{Information Acquisition Rate}
\label{sec:Acquisition_Rate}
One of the limitations of existing stochastic optimization methods is that they revisit the same points again and again, which results in a waste of resources, especially in population-based methods, in which they tend to cluster in regions that are already explored. This results in a decay of the effective information rate $I_{\text{eff}}(t)$ to zero. In contrast S-LCG guarantee no decay in information rate by it's structural design. The following proposition proves that S-LCG has constant information acquisition rate.

\begin{proposition}[Constant Information Rate]\label{prop:inforate} Let $I(k)$ be the number of new, unique $d$-dimensional points introduced by the $k$-th generator. Then for all $k\ge 1$, it have $I(k)=2n$. \end{proposition} 

\begin{proof}  By Corollary \ref{cor: zero_redundancy}, the $k$-th generator generates a sequence of $2n$ state is entirely disjoint from all previously studied cycles. These $2n$ states correspond to $2n$ different spatial coordinates that still need to be evaluated, since the bit-splitting encoding (see Eq.~\eqref{eq:bitsplit}) is a strict bijection. So \(I(k) = 2n\) is valid globally over the whole search space.
\end{proof}
Thus, after traversing $K$ generators, the algorithm has tested exactly $2n \times K$ different points. The $K$-th generator is as informative as the first one. The constant rate of information is important for high-dimensional scalability, since it ensures that the algorithm’s ability to explore does not artificially decrease with an exponential increase in the spatial volume due to localised convergence.

\section{The S-LCG Algorithm}
\label{sec:algo}




In this section, we describe all the major components of the proposed algorithm. We first describe the fixed and tunable parameters of the algorithm, and then we move on to elaborate on the two-level nested loop architecture introduced in Section~\ref{sec:surrogate}, which implicitly constructs and optimises the surrogate landscape. Next, we elaborate on the adaptive step control strategy that enables us to navigate the generator space effectively and attain an optimal value without having to examine each generator in G.

\subsection{Initialization}
\label{subsec:init}

    The procedure needs two things: the objective function $f(\mathbf{x})$ and the search domain boundaries $[L_j, U_j]$ for each variable $j = 1, \ldots, d$.
 
    \medskip
    \noindent\textbf{Fixed parameters:} The modulus $m$, increment $c$, and initial seed $x_0$ are established using the $S$-$LCG$ framework (Section~\ref{sec:lcg}).
    The exploitation parameter is set at $\delta = 2$ (the smallest even increase), and the sequence length for each generator is $2n$.

    \noindent\textbf{Tunable parameters:} Three parameters control how an algorithm works, ordered by sensitivity:
 
    \begin{enumerate}
    \item \textbf{Exploration step $\Delta$} \textit{(primary):} Regulates search resolution and exerts the most significant impact on performance.
          The derivation involves concatenating the binary representations of the variables, with the first variable occupying the least significant bits (LSB) and the last variable occupying the most significant bits (MSB):
          \begin{equation}
              \Delta \;=\; \operatorname{int}\!\left(
                  \operatorname{bin}(x_n,\, b_n) \;\|\; \cdots \;\|\;
                  \operatorname{bin}(x_1,\, b_1)
              \right).
              \label{eq:Delta}
          \end{equation}
          For example, if you have three variables with bit lengths of 4, 4, and 3, and you set \( x_1 = 0 \) and give the other bits to \texttt{00010001}, the 11-bit string you get is \texttt{00010001000}, which means \( \Delta = 136 \).
 
    \item \textbf{Stagnation threshold $S_{\max}$} \textit{(secondary):}
Count of consecutive valid generators before a mandated diversification leap without global enhancement.
 
    \item \textbf{Exploitation budget $E_{\max}$} \textit{(secondary):} The maximum $\delta$-steps that can be taken in a row during a local exploitation surge.
\end{enumerate}
 
The parameters $S_{\max}$ and $E_{\max}$ need only minor adjustments for different issue instances; $\Delta$ is the only parameter that usually needs to be calibrated for each situation.
\subsection{Two-level nested loop architecture and adaptive step control}

The central idea of the two-level nested loop architecture is to dictate what needs to be evaluated (the inner loop) and where to search next (the outer loop).

\paragraph{Inner loop (deterministic sequence evaluation):}The inner loop executes the recursive $S$-$LCG$ function for a given generator $\alpha$ for a fixed $2n$ times. The loop performs bit-split encoding on each generated integer point~\ref{sec:bit-split} and then evaluates the objective function $f$ at it. In this way, for each generated cycle, the inner loop finds the local minimum within it. Hence, the inner loop is a complete, fixed cost evaluation of a single, disjoint orbit and does not involve any heuristics. The deterministic rigidity of the loop implies that the zero-redundancy property (Theorem~\ref{thm:unique_rep}) is always achieved.

\paragraph{Outer loop (Adaptive generator navigation):}The algorithm's optimisation intelligence comes from the outer loop, as it determines which generator to evaluate next in the generator space. The generator is chosen based on the step size, which in turn is determined by two mechanisms: a fine exploitation step $\delta$ and a coarse exploration step $\Delta$. The decision logic for the step size is based on local surrogate descent and stagnation detection. It traverses the generator space by evaluating a three-tier priority hierarchy described as follows:

\paragraph{Priority 1-Stagnation detection:}Each time the successive generator fails to yield an improvement over the global best objective value found so far, the stagnation counter is incremented. This metric helps to track exactly how long the algorithm has searched a specific region without success. If the stagnation counter exceeds $S_{max}$, a $\Delta$ step exploratory jump is mandated, and the counter is reset (Algorithm 1, line 24). Stagnation detection strictly overrides exploitation; ongoing global stagnation forces diversification regardless of local descent. This prevents the algorithm from endlessly refining a sub-optimal basin.

\paragraph{Priority 2-Local exploitation:}Exploitation corresponds to moving to an adjacent generator that will yield a structurally correlated set justified by surrogate smoothness (Section 4.2). Therefore, if the current generator yields an improved local optimum compared to the previous generator ($f^*_{\text{local}} < f_{\text{prev}}$) and the exploitation counter is below $E_{\max}$, a $\delta$-step is executed (line 26). The role of $E_{\max}$ is to bound the depth of the exploitation surge and prevent infinite local trapping.

\begin{remark}

The exploitation trigger deliberately compares $f^*_{\text{local}}$ against $f_{\text{prev}}$ (the previous generator's local optimum), rather than $f^*_{\text{global}}$. This determines whether the surrogate $g(\alpha)$ is \emph{locally descending} in the current region of seed space, enabling the exploitation of productive neighborhoods even if the global optimum currently resides elsewhere.

\end{remark}

\paragraph{Priority 3 -- Default exploration:} 

If neither the stagnation nor the local improvement thresholds are met, a $\Delta$-step advance occurs, and the exploitation counter resets (line 29). This default diversification ensures continuous progression through the seed space.

\begin{figure}[H]
\small
\centering
\begin{tikzpicture}[font=\small]
 
    \draw[->, thick] (0,0) -- (14,0) node[right] {Generator seed $\alpha$};
 
    \foreach \x/\lbl in {0/$\alpha_0$, 2/$\alpha_1$, 5.5/$\alpha_k$,
                          7.5/$\alpha_
{k+1}$, 9.5/$\alpha_{k+2}$, 12.5/$\alpha_{k+j}$}
        \draw (\x, 0.08) -- (\x,-0.08) node[below] {\lbl};
 
    \draw[->, thick, blue!70!black, bend left=40]
        (0,0.2) to node[above, midway] {$\Delta$-step (explore)} (5.5,0.2);
 
    \foreach \from/\to in {5.5/7.5, 7.5/9.5}
        \draw[->, thick, red!70!black]
            (\from,0.2) to node[above, midway] {$\delta$} (\to,0.2);
 
    \draw[->, thick, blue!70!black, bend left=35]
        (9.5,0.2) to node[above, midway] {$\Delta$-step (explore)} (12.5,0.2);
 
    \node[align=center, text=red!70!black, font=\footnotesize] at (7.5,-0.8)
        {Exploitation burst\\($\delta$-steps, $g(\alpha)$ descending)};
 
    \node[align=center, text=blue!70!black, font=\footnotesize] at (2.75,-0.8)
        {Diversification\\($\Delta$-step)};
 
    \node[align=center, text=blue!70!black, font=\footnotesize] at (11.0,-0.8)
        {Diversification\\($\Delta$-step)};
 
    \foreach \x in {0, 5.5, 7.5, 9.5, 12.5}
        \filldraw[black] (\x,0) circle (2.5pt);
 
    \foreach \x in {2, 3.5}
        \draw[gray, densely dashed] (\x,0.08) -- (\x,-0.08);
    \node[gray, font=\footnotesize] at (3,-0.2) {skipped};
 
\end{tikzpicture}
\caption{%
How to get around in the generator space in $S$-$LCG$.
    Extensive $\Delta$-steps (blue) improve the search by moving to generators that are far away and not structurally related.
    Minor $\delta$-steps (red) take advantage of a good area when the surrogate $g(\alpha)$ shows local descent.
    Gray dashes show seeds that \textsc{IsGenerator} found to be faulty and don't need function evaluations.
    Every seed that is evaluated (filled circle) starts a full $2n$-evaluation inner loop.
}
\label{fig:gennav}
\end{figure}
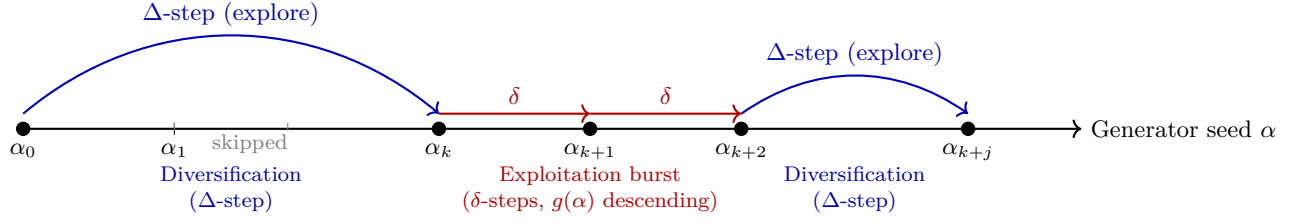

\subsection{Complete Algorithm}
\label{sec:description}
 
The complete S-LCG procedure is given in Algorithm~\ref{alg:slcg}
    \begin{algorithm}[H]
    \caption{S-LCG Optimization Algorithm}
    \label{alg:slcg}
    \begin{algorithmic}[1]\label{algorithm}
    \footnotesize
    \renewcommand{\baselinestretch}{1}\selectfont
    \Require Objective function $f(\mathbf{x})$, search domain $\Omega$
    \Ensure Best solution $\mathbf{x}^*$ and fitness $f^* = f(\mathbf{x}^*)$

    \textit{\textcolor{gray}{\% Phase 1: Initialization}}
    \State Set constants: modulus $m$, increment $c$, seed $x_0$
    \State Set parameters: $\alpha_0$, $\alpha_{\max}$, $\Delta_{\text{step}}$, $\delta_{\text{step}}$, $S_{\max}$, $E_{\max}$
    \State $f^*_{\text{global}} \gets +\infty$; \; $\mathbf{x}^*_{\text{global}} \gets \texttt{null}$
    \State $f_{\text{prev}} \gets +\infty$; \; $\alpha_k \gets \alpha_0$
    \State $\texttt{stag\_ctr} \gets 0$; \; $\texttt{exploit\_ctr} \gets 0$


\textit{\textcolor{gray}{\% Phase 2: Main Search Loop}}
\While{$\alpha_k \leq \alpha_{\max}$}
    \If{\Call{IsGenerator}{$\alpha_k$}} \Comment{Validate multiplier}\ref{alg:generator_check}

        \textit{\textcolor{gray}{\quad\quad\% Phase 2a: Sequence Evaluation}}
        \State $f^*_{\text{local}} \gets +\infty$; \; $\mathbf{x}^*_{\text{local}} \gets \texttt{null}$
        \State $\mathbf{x} \gets x_0$ \Comment{Initialize LCG state}
        \For{$i = 0$ \textbf{to} $2n - 1$}
            \If{$f(\mathbf{x}) < f^*_{\text{local}}$}
                \State $f^*_{\text{local}} \gets f(\mathbf{x})$; \; $\mathbf{x}^*_{\text{local}} \gets \mathbf{x}$
            \EndIf
            \State $\mathbf{x} \gets (\alpha_k \cdot \mathbf{x} + c) \bmod m$ \Comment{S-LCG transition}
        \EndFor

        \textit{\textcolor{gray}{\quad\quad\% Phase 2b: Global Memory Update}}
        \If{$f^*_{\text{local}} < f^*_{\text{global}}$}
            \State $f^*_{\text{global}} \gets f^*_{\text{local}}$; \; $\mathbf{x}^*_{\text{global}} \gets \mathbf{x}^*_{\text{local}}$
            \State $\texttt{stag\_ctr} \gets 0$
        \Else
            \State $\texttt{stag\_ctr} \gets \texttt{stag\_ctr} + 1$
        \EndIf

        \textit{\textcolor{gray}{\quad\quad\% Phase 2c: Adaptive Step Control}}
        \If{$\texttt{stag\_ctr} > S_{\max}$} \Comment{Stagnation $\Rightarrow$ force exploration}
            \State $\alpha_k \gets \alpha_k + \Delta_{\text{step}}$
            \State $\texttt{stag\_ctr} \gets 0$
        \ElsIf{$\texttt{exploit\_ctr} < E_{\max}$ \textbf{and} $f^*_{\text{local}} < f_{\text{prev}}$}
            \State $\alpha_k \gets \alpha_k + \delta_{\text{step}}$ \Comment{Exploit nearby multiplier}
            \State $\texttt{exploit\_ctr} \gets \texttt{exploit\_ctr} + 1$
        \Else \Comment{No improvement $\Rightarrow$ explore}
            \State $\alpha_k \gets \alpha_k + \Delta_{\text{step}}$
            \State $\texttt{exploit\_ctr} \gets 0$
        \EndIf

        \State $f_{\text{prev}} \gets f^*_{\text{local}}$

    \Else \Comment{Invalid multiplier $\Rightarrow$ skip}
        \State $\alpha_k \gets \alpha_k + \Delta_{\text{step}}$
    \EndIf
\EndWhile


\State \Return $\mathbf{x}^*_{\text{global}}$, $f^*_{\text{global}}$

\end{algorithmic}
\end{algorithm}

\subsection{Computational Complexity and Memory Efficiency}
\label{subsec:complexity}

\paragraph{Time complexity:} The \textsc{IsGenerator} check requires $O(n)$ operations per outer loop iteration. Valid generators subsequently trigger the inner loop, demanding exactly $2n$ function evaluations and corresponding $O(n)$ bit-splitting operations. Therefore, evaluating a single valid generator costs $O(n \cdot f_{\text{eval}})$. Invalid generators incur only the $O(n)$ check before being discarded. For $K$ total evaluated generators, the overarching time complexity is tightly bounded at $O(K \cdot n \cdot f_{\text{eval}})$.

\paragraph{Space complexity:}
A profound advantage of S-LCG is its minimal memory footprint. The algorithm stores only the optimal solution vector $\mathbf{x}^*_{\text{global}} \in \mathbb{R}^d$, the scalar fitness $f^*_{\text{global}}$, and three tracking counters ($\texttt{stag\_ctr}$, $\texttt{exploit\_ctr}$, $f_{\text{prev}}$). The LCG state $\mathbf{x}$ is a single integer updated in-place. Because S-LCG maintains no population matrices, velocity arrays, or covariance matrices, its space complexity is strictly $O(d)$. This dimension-independent memory profile renders S-LCG exceptionally scalable compared to standard population-based methods, which require $O(P \cdot d)$ or $O(d^2)$ storage (Table~\ref{tab:complexity}).
\begin{table}[H]
\small
\centering
\footnotesize 
\setlength{\tabcolsep}{4pt} 
\renewcommand{\arraystretch}{0.9} 

\begin{tabular}{lcc}
\toprule
\textbf{Algorithm} & \textbf{Time (per iteration)} & \textbf{Space} \\
\midrule
S-LCG (this work) & $O(n \cdot f_{\text{eval}})$   & $O(d)$        \\
GA / DE / PSO      & $O(P \cdot f_{\text{eval}})$   & $O(P \cdot d)$\\
CMA-ES             & $O(d^2 + f_{\text{eval}})$     & $O(d^2)$      \\
Simulated Annealing& $O(f_{\text{eval}})$           & $O(d)$        \\
\bottomrule
\end{tabular}
\caption{Complexity comparison of S-LCG against representative meta-heuristics.}
\label{tab:complexity}
\end{table}

\section{Experimental Evaluation}\label{sec:experiments}
To validate S-LCG's capability to compete with well-established algorithms, this section  evaluates 26 standard benchmark functions scaled from 2 to 30 dimensions, along with three well-known constrained engineering design problems. Against eight binary optimization algorithms, and are supported by statistical analysis of the results, as discussed in the following subsections.

\subsection{Experimental Setup}\label{sec:setup}
    To evaluate $S$-$LCG$, a well-known standardized benchmark suite in the meta-heuristic literature is employed, consisting of 26 classical test functions widely used over the past 25 years ~\cite{Yao1999,Yang2010TestProblems,GWO,Ahmadianfar2021RUN}. These functions can be categorized into three categories: The first category consists of seven unimodal functions ($f_1$-$f_7$), the next are nine multimodal functions ($f_8$-$f_{16}$), and the rest are the ten fixed-dimension multimodal functions ($f_{17}$-$f_{26}$).
    This suite of benchmark functions was carefully chosen for its wide range of structures, not just because it's the norm. Each of these functions has its own level of difficulty. Even among the unimodal functions, each tests different abilities.
    The performance of S-LCG is compared with eight state-of-the-art binary metaheuristics from the literature: The Genetic Algorithm (GA), Binary Ant Colony Optimiser (BACO), Binary Differential Evolution (BDE), Binary Grey Wolf Optimiser (BGWO), two variants of the Binary Particle Swarm Optimiser, one using the sigmoid transfer function (BPSO) and the other using the TANH (T-BPSO), Binary Simulated Annealing (BSA), and Binary Tabu Search (BTS). The implementation of each competing algorithm is based on the parameter recommendations of the original source. For the stochastic variability, every algorithm is independently executed 30 times per problem, while the deterministic S-LCG is run only once. The parameters of S-LCG are fixed for all problems in the trials: $\delta = 2$ for the exploitation step, $E_{\max} = 60$ for the evaluation budget, and $S_{\max} = 5000$ for the stagnation threshold. The exploration step $\Delta$ was dynamically determined from the problem's dimensionality by bit concatenation as explained in the previous section ~\ref{sec:algo}.

    \textit{Because of the page limit, the full performance tables and charts (Minimum, Maximum, Mean, Standard Deviation) for all nine algorithms for all twenty-six functions and eight dimensions are contained in the Supplementary Material. Then review the statistical analysis discussed in the following sections.}

\subsection{Benchmark Performance and Scalability}\label{sec:benchmark_results}

The performance of S-LCG shows an evident and highly beneficial trend, in which its relative advantage increases with the increase of dimension, as it has been claimed in section~\ref{sec:Acquisition_Rate}, where the acquisition rate is discussed. Therefore, the following discusses the most significant result.

\textbf{Unimodal Exploitation (F1–F7):} Unimodal landscapes well known for it's exploitation ability of an algorithm. S-LCG achieves accuracy close to the machine precision for the Sphere (F1) and Schwefel 2.22 (F2) functions. Remarkably, this accuracy holds perfectly as the dimensionality increases to $d=30$. The iconic Rosenbrock valley (F5) causes a drastic drop in the performance of population-based methods in higher dimensions, due to the presence of a difficult valley (e.g., GA mean degradation to $2.12 \times 10^2$ for $d=30$). In contrast, S-LCG has an excellent $5.99 \times 10^{-2}$, indicating that it can reliably traverse narrow, curved valleys without getting trapped due to its deterministic geometric sampling for the space uniformly as discussed in section ~\ref{sec:theory}.

\textbf{Multi-modal Exploration (F8–F16):} Global search procedures are evaluated on multimodal functions. On the deceptive Schwefel function (F8), the global optimum is far away from the local attractors and binary algorithms used to suffer from severe premature convergence in higher dimensions. S-LCG ahcived $-1.26 \times 10^4$, which implies that the $\Delta$-step exploration of S-LCG always samples isolated optimal basins. The same robustness is found on the highly dense Rastrigin (F9) and Ackley (F10) functions. The only Function S-LCG performed poorly is F14, which needs careful tuning to find a better solution. For functions with shrinking optimal basins (F15, F16), stochastic methods completely fail at higher dimensions, whereas S-LCG correctly locks on the global minimum.
 
The results for the scalable functions for $d=30$ are shown in Table ~\ref{tab:results_mean_d30}. Statistical significance is evaluated using the one-sample Wilcoxon signed-rank test between the value obtained with S-LCG and the median of 30 runs of the opposing algorithms. The algorithms column is comparatively ranked using Friedman’s test. To track convergence and show the exploitation analysis, it is shown in figure ~\ref{fig:visual_unimodal}, where using an array to save the exploitation makes the algorithm capable of detecting the existence of multiple global minima.

\newcolumntype{C}{>{\centering\arraybackslash\scriptsize}p{1.5cm}}
\small
\setlength{\tabcolsep}{3pt}
\renewcommand{\arraystretch}{0.9}
\begin{longtable}{lCCCCCCCCC}

\toprule
\textbf{Function} & \multicolumn{1}{c}{\textbf{S-LCG}} & \multicolumn{1}{c}{\textbf{GA}} & \multicolumn{1}{c}{\textbf{BACO}} & \multicolumn{1}{c}{\textbf{BSA}} & \multicolumn{1}{c}{\textbf{BTS}} & \multicolumn{1}{c}{\textbf{T-BPSO}} & \multicolumn{1}{c}{\textbf{BGWO}} & \multicolumn{1}{c}{\textbf{BPSO}} & \multicolumn{1}{c}{\textbf{BDE}} \\
\midrule
\endfirsthead
\toprule
\textbf{Function} & \multicolumn{1}{c}{\textbf{S-LCG}} & \multicolumn{1}{c}{\textbf{GA}} & \multicolumn{1}{c}{\textbf{BACO}} & \multicolumn{1}{c}{\textbf{BSA}} & \multicolumn{1}{c}{\textbf{BTS}} & \multicolumn{1}{c}{\textbf{T-BPSO}} & \multicolumn{1}{c}{\textbf{BGWO}} & \multicolumn{1}{c}{\textbf{BPSO}} & \multicolumn{1}{c}{\textbf{BDE}} \\
\midrule
\endfirsthead

\bottomrule
\multicolumn{10}{r}{\textit{Continued on next page}} \\
\endfoot

\bottomrule
\noalign{\vspace{3pt}} 
\caption{Mean Results on Benchmark Functions ($d=30$)}\label{tab:results_mean_d30} \\
\endlastfoot

F1 & \phantom{-}2.66e-07\phantom{$^{+}$} & \phantom{-}2.71e-07$^{=}$ & \phantom{-}1.49e-06$^{=}$ & \phantom{-}2.66e-07$^{=}$ & \phantom{-}5.87e-07$^{=}$ & \phantom{-}1.76e+04$^{+}$ & \phantom{-}2.16e+04$^{+}$ & \phantom{-}2.65e+04$^{+}$ & \phantom{-}3.52e+04$^{+}$ \\
F2 & \phantom{-}2.81e-04\phantom{$^{+}$} & \phantom{-}2.81e-04$^{=}$ & \phantom{-}3.84e-04$^{+}$ & \phantom{-}2.81e-04$^{=}$ & \phantom{-}3.64e-04$^{+}$ & \phantom{-}4.99e+01$^{+}$ & \phantom{-}5.08e+01$^{+}$ & \phantom{-}7.43e+01$^{+}$ & \phantom{-}8.11e+01$^{+}$ \\
F3 & \phantom{-}1.59e-04\phantom{$^{+}$} & \phantom{-}6.98e+03$^{+}$ & \phantom{-}2.34e+03$^{+}$ & \phantom{-}3.47e+04$^{+}$ & \phantom{-}3.67e+04$^{+}$ & \phantom{-}3.48e+04$^{+}$ & \phantom{-}2.89e+04$^{+}$ & \phantom{-}3.27e+04$^{+}$ & \phantom{-}4.01e+04$^{+}$ \\
F4 & \phantom{-}3.53e-03\phantom{$^{+}$} & \phantom{-}5.75e-02$^{+}$ & \phantom{-}5.82e-01$^{+}$ & \phantom{-}8.52e-04$^{-}$ & \phantom{-}1.11e-01$^{+}$ & \phantom{-}6.05e+01$^{+}$ & \phantom{-}7.13e+01$^{+}$ & \phantom{-}6.08e+01$^{+}$ & \phantom{-}6.71e+01$^{+}$ \\
F5 & \phantom{-}5.99e-02\phantom{$^{+}$} & \phantom{-}2.12e+02$^{+}$ & \phantom{-}1.27e+02$^{+}$ & \phantom{-}9.94e+03$^{+}$ & \phantom{-}2.85e+03$^{+}$ & \phantom{-}2.56e+07$^{+}$ & \phantom{-}5.11e+07$^{+}$ & \phantom{-}4.62e+07$^{+}$ & \phantom{-}7.75e+07$^{+}$ \\
F6 & \phantom{-}2.17e-01\phantom{$^{+}$} & \phantom{-}3.40e+00$^{+}$ & \phantom{-}3.68e+00$^{+}$ & \phantom{-}3.86e+00$^{+}$ & \phantom{-}3.87e+00$^{+}$ & \phantom{-}1.66e+04$^{+}$ & \phantom{-}2.14e+04$^{+}$ & \phantom{-}2.76e+04$^{+}$ & \phantom{-}3.55e+04$^{+}$ \\
F7 & \phantom{-}7.04e-05\phantom{$^{+}$} & \phantom{-}3.09e-03$^{+}$ & \phantom{-}5.75e-03$^{+}$ & \phantom{-}1.74e-02$^{+}$ & \phantom{-}3.01e-02$^{+}$ & \phantom{-}1.23e+01$^{+}$ & \phantom{-}2.31e+01$^{+}$ & \phantom{-}2.15e+01$^{+}$ & \phantom{-}3.51e+01$^{+}$ \\
F8 & -1.26e+04\phantom{$^{+}$} & -1.18e+04$^{+}$ & -1.24e+04$^{+}$ & -1.15e+04$^{+}$ & -9.63e+03$^{+}$ & -7.22e+03$^{+}$ & -6.76e+03$^{+}$ & -6.27e+03$^{+}$ & -6.66e+03$^{+}$ \\
F9 & \phantom{-}1.38e-07\phantom{$^{+}$} & \phantom{-}2.65e+01$^{+}$ & \phantom{-}1.72e+01$^{+}$ & \phantom{-}4.16e+01$^{+}$ & \phantom{-}7.26e+01$^{+}$ & \phantom{-}2.48e+02$^{+}$ & \phantom{-}2.07e+02$^{+}$ & \phantom{-}2.84e+02$^{+}$ & \phantom{-}3.11e+02$^{+}$ \\
F10 & \phantom{-}7.54e-05\phantom{$^{+}$} & \phantom{-}7.62e-05$^{+}$ & \phantom{-}1.62e-04$^{+}$ & \phantom{-}7.54e-05$^{=}$ & \phantom{-}6.20e-01$^{+}$ & \phantom{-}1.39e+01$^{+}$ & \phantom{-}1.40e+01$^{+}$ & \phantom{-}1.56e+01$^{+}$ & \phantom{-}1.65e+01$^{+}$ \\
F11 & \phantom{-}6.52e-07\phantom{$^{+}$} & \phantom{-}4.05e-02$^{+}$ & \phantom{-}3.53e-02$^{+}$ & \phantom{-}1.35e-01$^{+}$ & \phantom{-}1.42e-01$^{+}$ & \phantom{-}1.47e+02$^{+}$ & \phantom{-}1.98e+02$^{+}$ & \phantom{-}2.43e+02$^{+}$ & \phantom{-}3.20e+02$^{+}$ \\
F12 & \phantom{-}1.06e-04\phantom{$^{+}$} & \phantom{-}5.33e-01$^{+}$ & \phantom{-}5.00e-01$^{+}$ & \phantom{-}5.15e-01$^{+}$ & \phantom{-}6.24e-01$^{+}$ & \phantom{-}3.34e+07$^{+}$ & \phantom{-}8.90e+07$^{+}$ & \phantom{-}5.30e+07$^{+}$ & \phantom{-}1.15e+08$^{+}$ \\
F13 & \phantom{-}6.42e-03\phantom{$^{+}$} & \phantom{-}1.68e+00$^{+}$ & \phantom{-}1.69e+00$^{+}$ & \phantom{-}1.59e+00$^{+}$ & \phantom{-}1.87e+00$^{+}$ & \phantom{-}1.00e+08$^{+}$ & \phantom{-}1.85e+08$^{+}$ & \phantom{-}1.61e+08$^{+}$ & \phantom{-}3.01e+08$^{+}$ \\
F14 & -1.33e+01\phantom{$^{+}$} & -2.74e+01$^{-}$ & -2.78e+01$^{-}$ & -2.70e+01$^{-}$ & -2.35e+01$^{-}$ & -1.31e+01$^{=}$ & -1.39e+01$^{-}$ & -1.16e+01$^{+}$ & -1.10e+01$^{+}$ \\
F15 & -1.00e+00\phantom{$^{+}$} & \phantom{-}5.13e-254$^{+}$ & \phantom{-}1.95e-251$^{+}$ & \phantom{-}1.02e-75$^{+}$ & \phantom{-}4.86e-209$^{+}$ & \phantom{-}2.79e-115$^{+}$ & \phantom{-}1.92e-255$^{+}$ & \phantom{-}2.50e-99$^{+}$ & \phantom{-}3.77e-124$^{+}$ \\
F16 & -9.99e-01\phantom{$^{+}$} & \phantom{-}5.95e-13$^{+}$ & \phantom{-}5.21e-13$^{+}$ & \phantom{-}5.22e-11$^{+}$ & \phantom{-}7.19e-13$^{+}$ & \phantom{-}9.62e-11$^{+}$ & \phantom{-}3.15e-10$^{+}$ & \phantom{-}3.33e-10$^{+}$ & \phantom{-}1.87e-09$^{+}$ \\
\midrule
\textbf{+/=/-} & \multicolumn{1}{c}{-} & \multicolumn{1}{c}{13/2/1} & \multicolumn{1}{c}{14/1/1} & \multicolumn{1}{c}{11/3/2} & \multicolumn{1}{c}{14/1/1} & \multicolumn{1}{c}{15/1/0} & \multicolumn{1}{c}{15/0/1} & \multicolumn{1}{c}{16/0/0} & \multicolumn{1}{c}{16/0/0} \\
\end{longtable}

\bgroup
\renewcommand\LTcaptype{figure} 

\begin{longtable}{cccc}

\textbf{Exploitation History} & \textbf{Convergence} & \textbf{Exploitation History} & \textbf{Convergence} \\[2ex]

\includegraphics[width=0.23\textwidth]{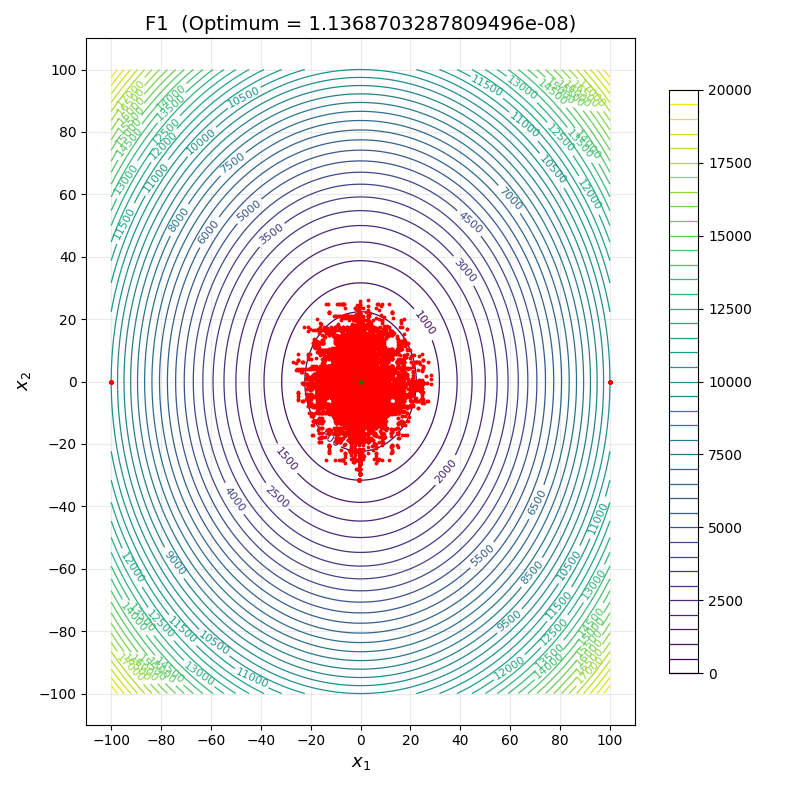} &
\includegraphics[width=0.25\textwidth]{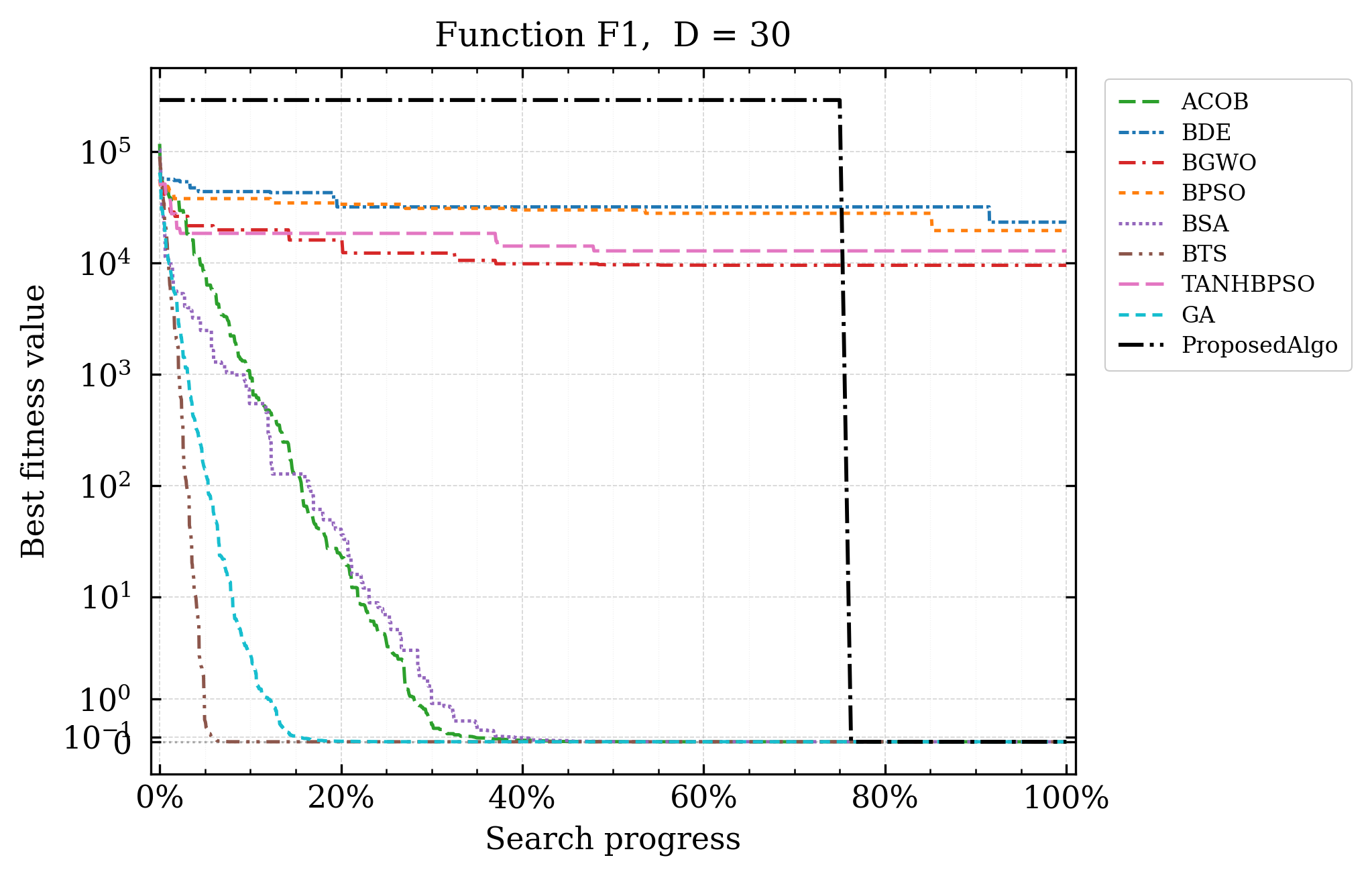} & 
\includegraphics[width=0.23\textwidth]{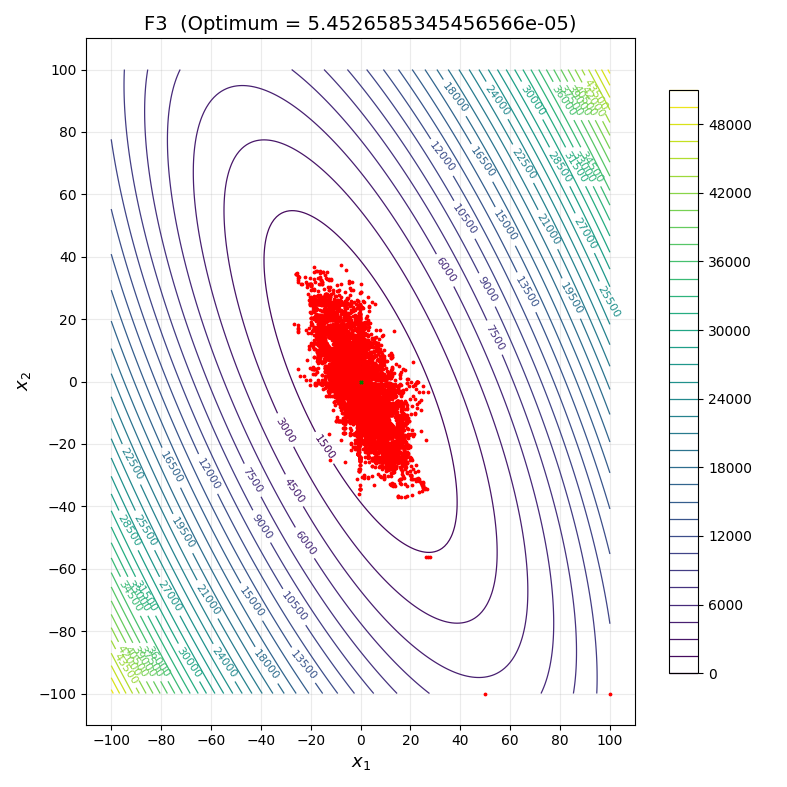} &
\includegraphics[width=0.25\textwidth]{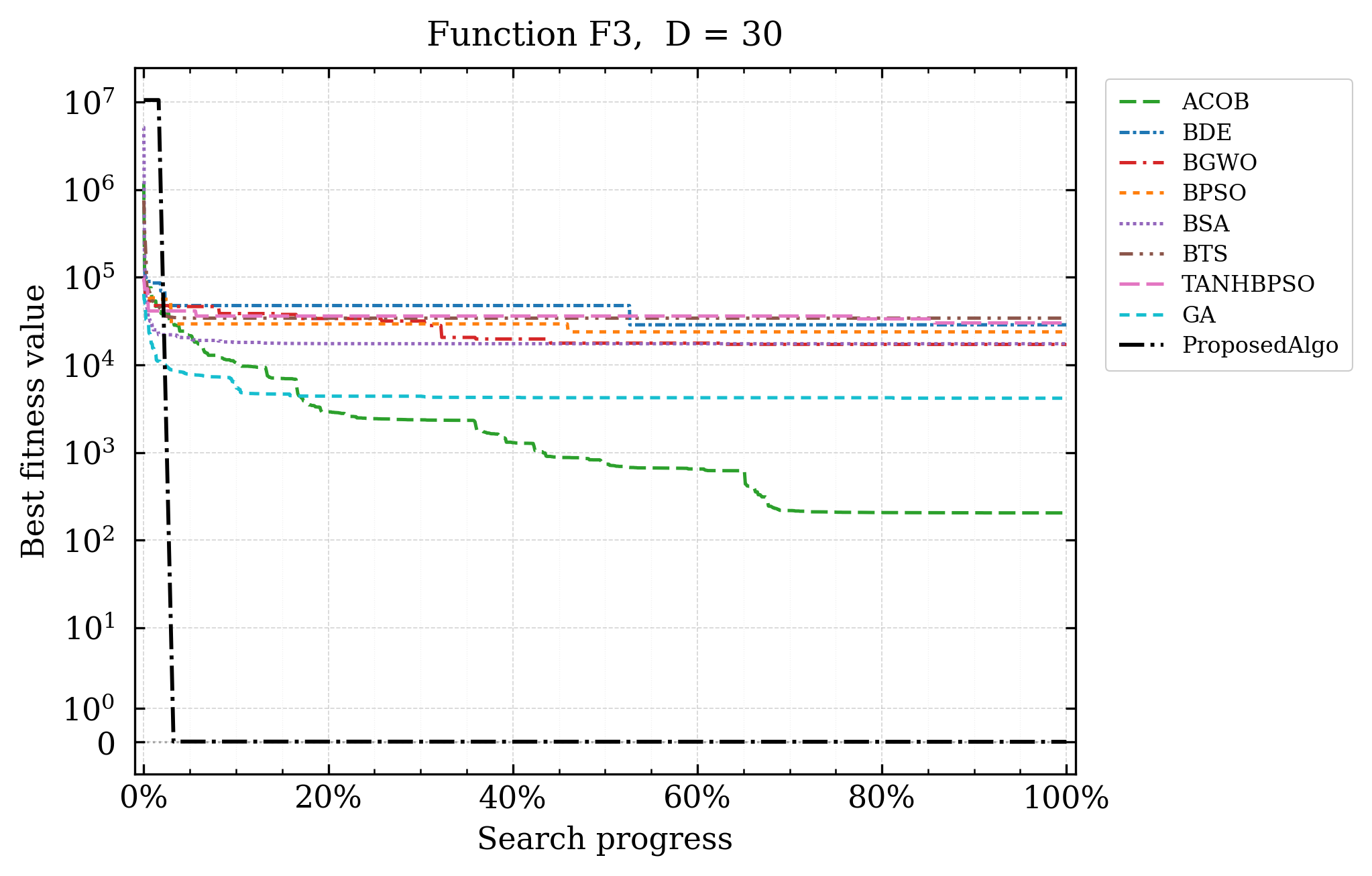} \\[2ex]

\includegraphics[width=0.23\textwidth]{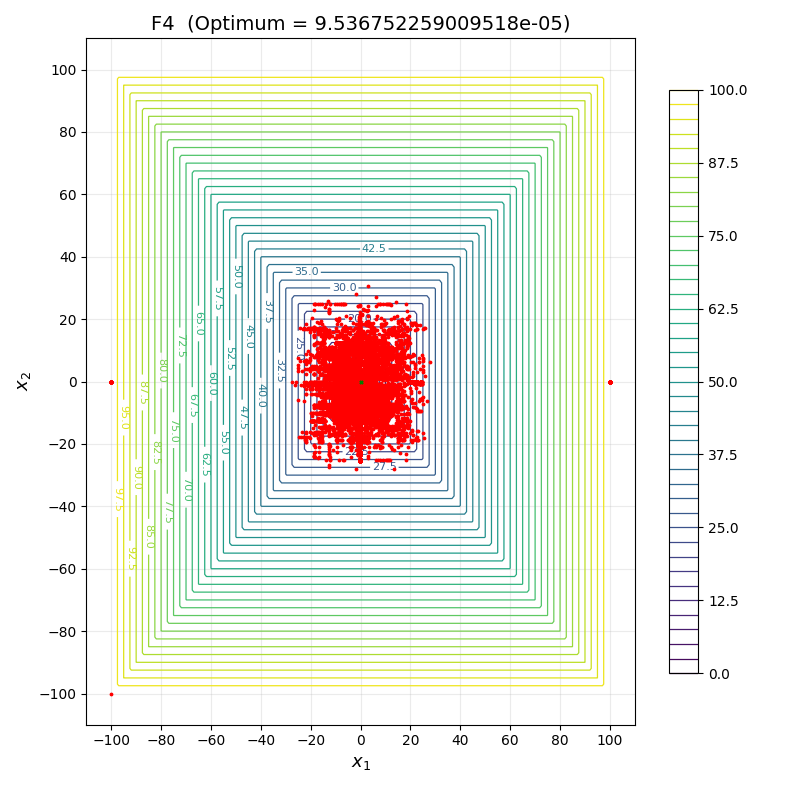} &
\includegraphics[width=0.25\textwidth]{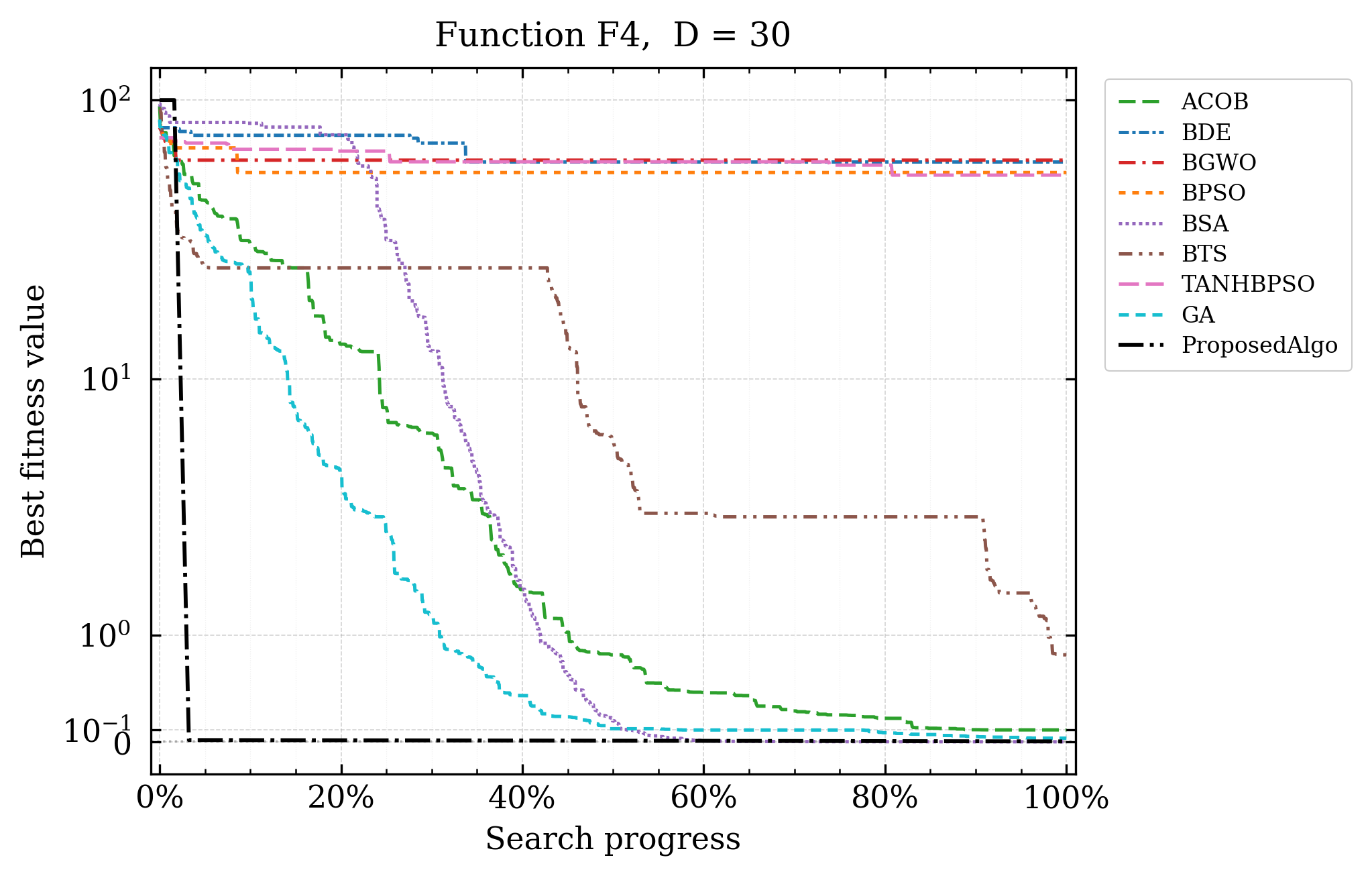} & 
\includegraphics[width=0.23\textwidth]{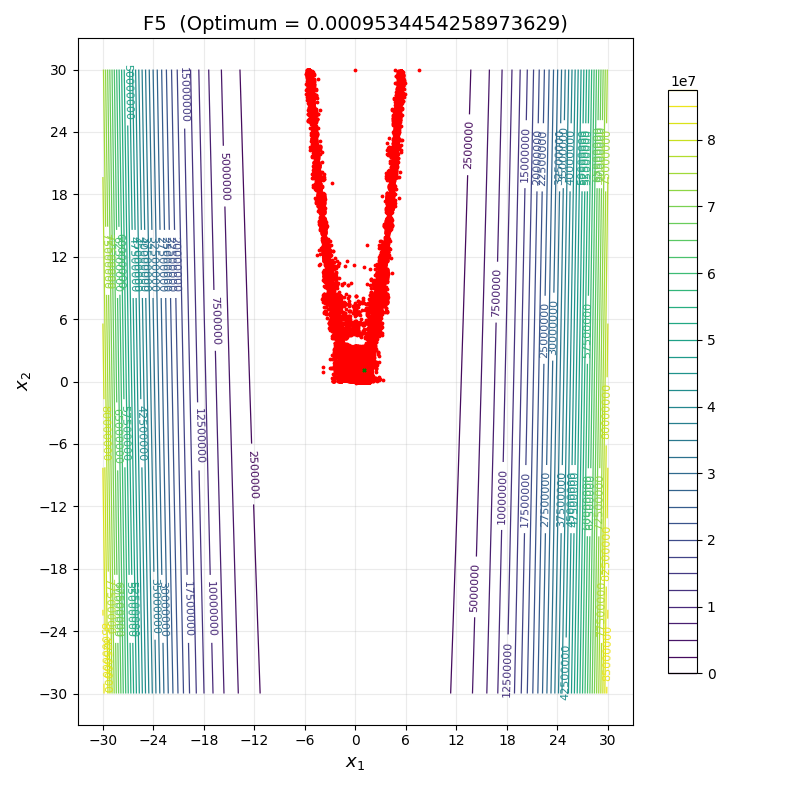} &
\includegraphics[width=0.25\textwidth]{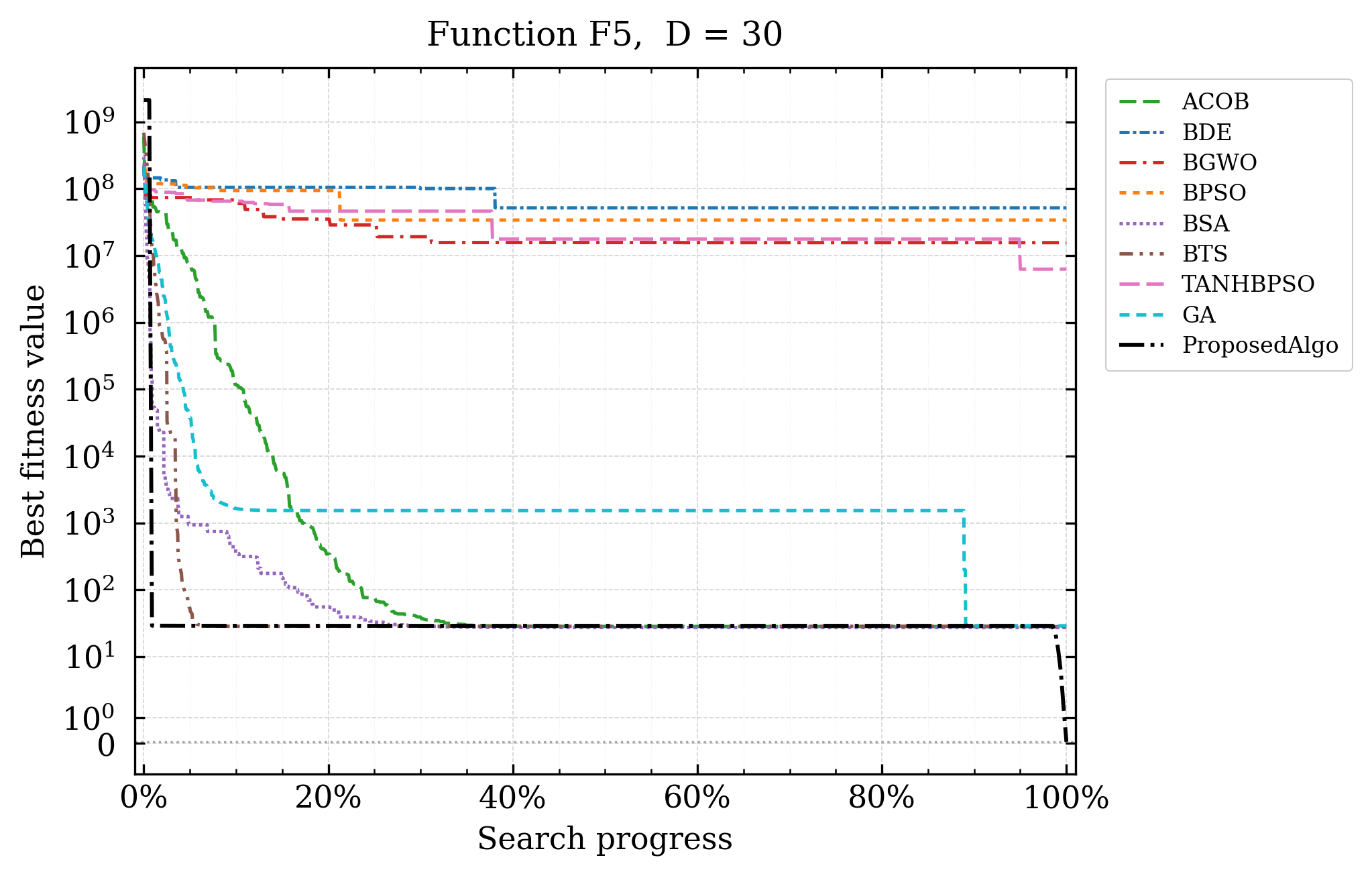} \\[2ex]

\includegraphics[width=0.23\textwidth]{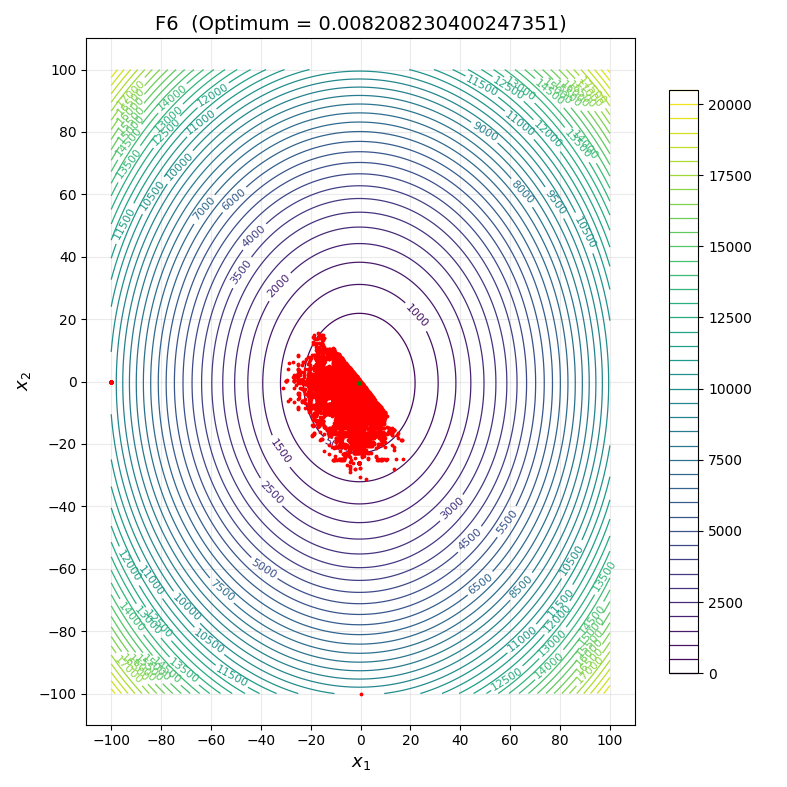} &
\includegraphics[width=0.25\textwidth]{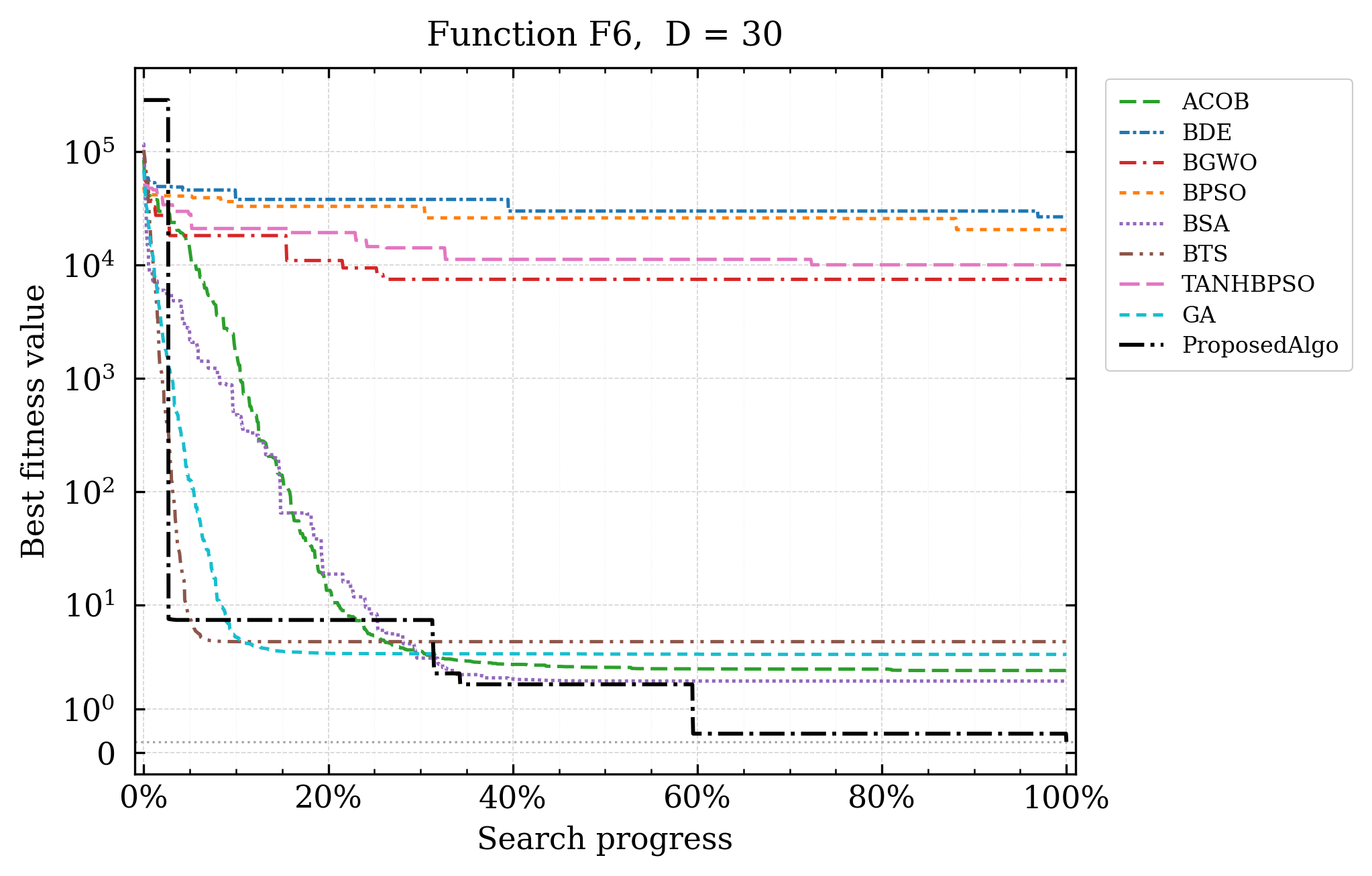} & 
\includegraphics[width=0.23\textwidth]{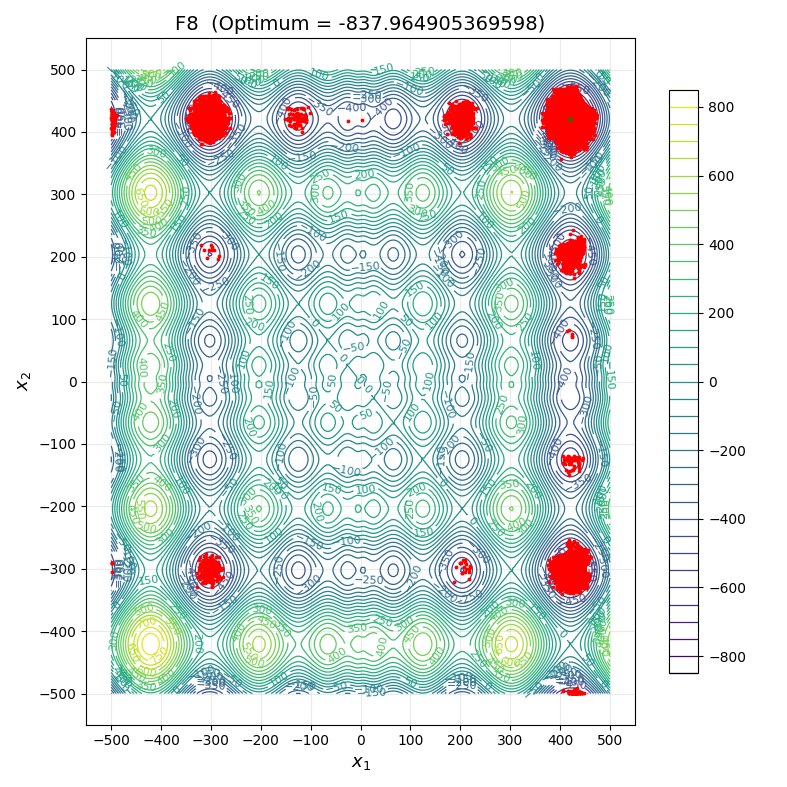} &
\includegraphics[width=0.25\textwidth]{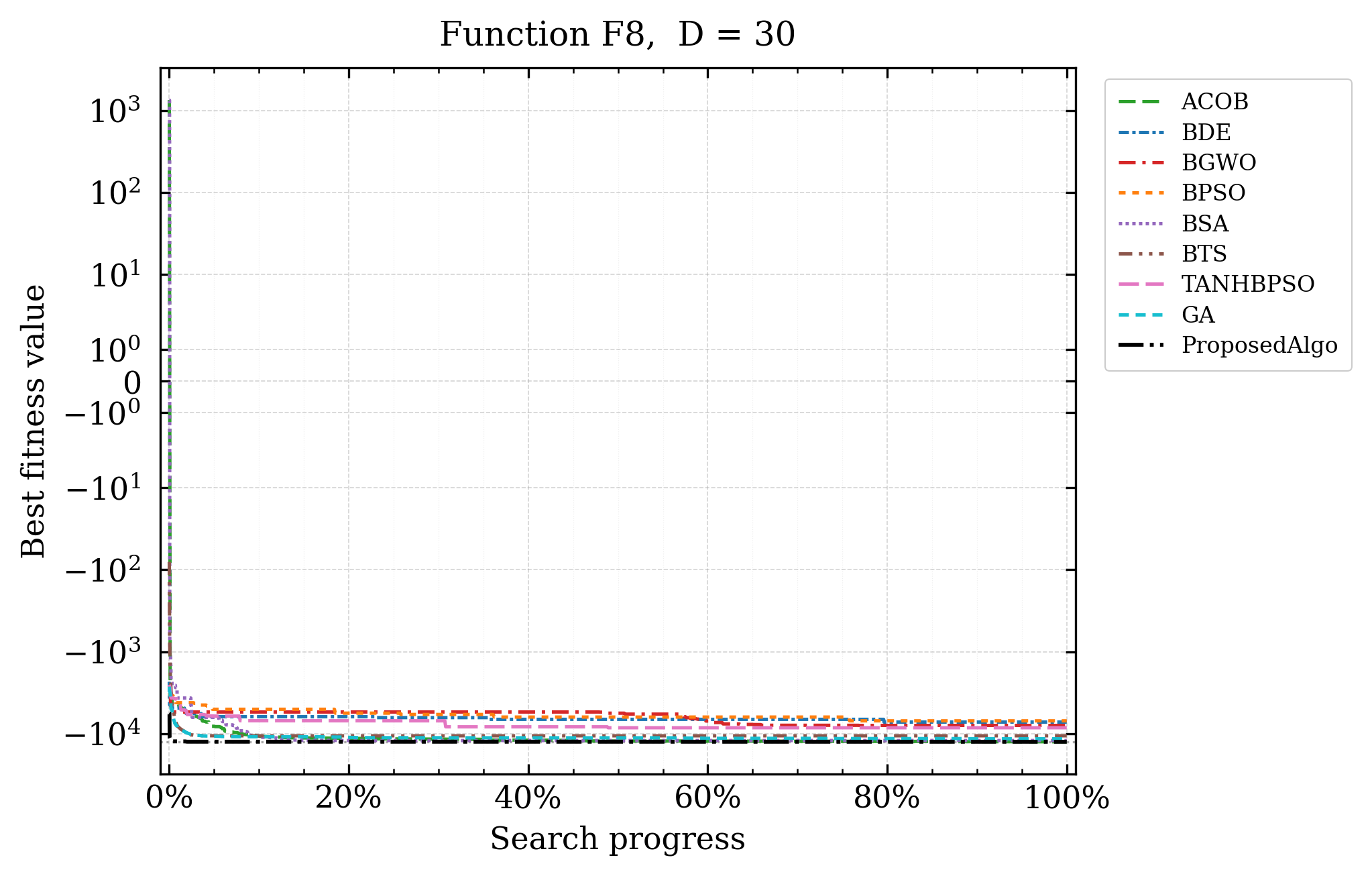} \\

\includegraphics[width=0.23\textwidth]{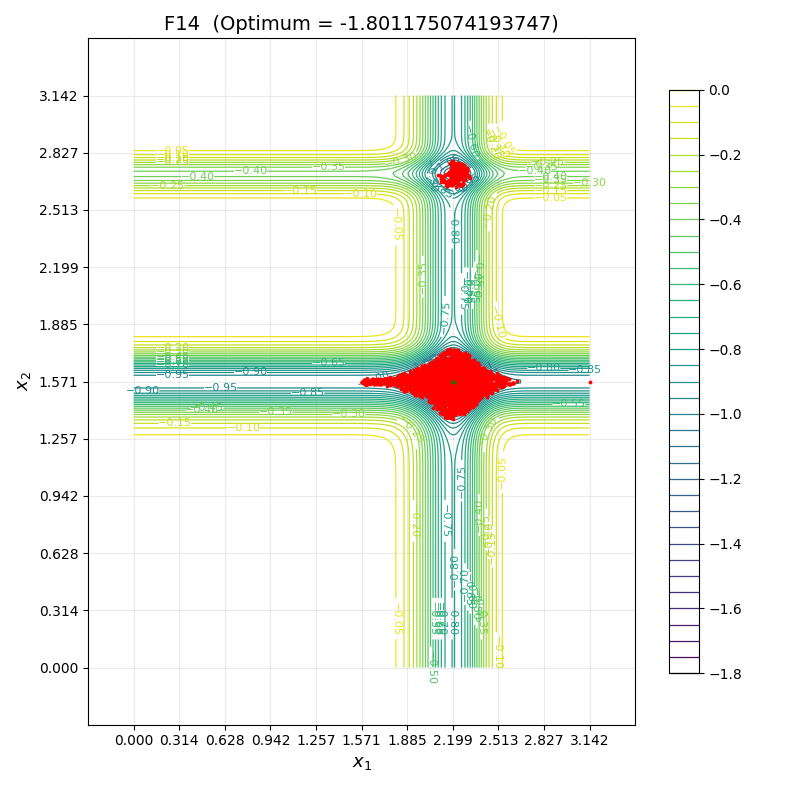} &
\includegraphics[width=0.25\textwidth]{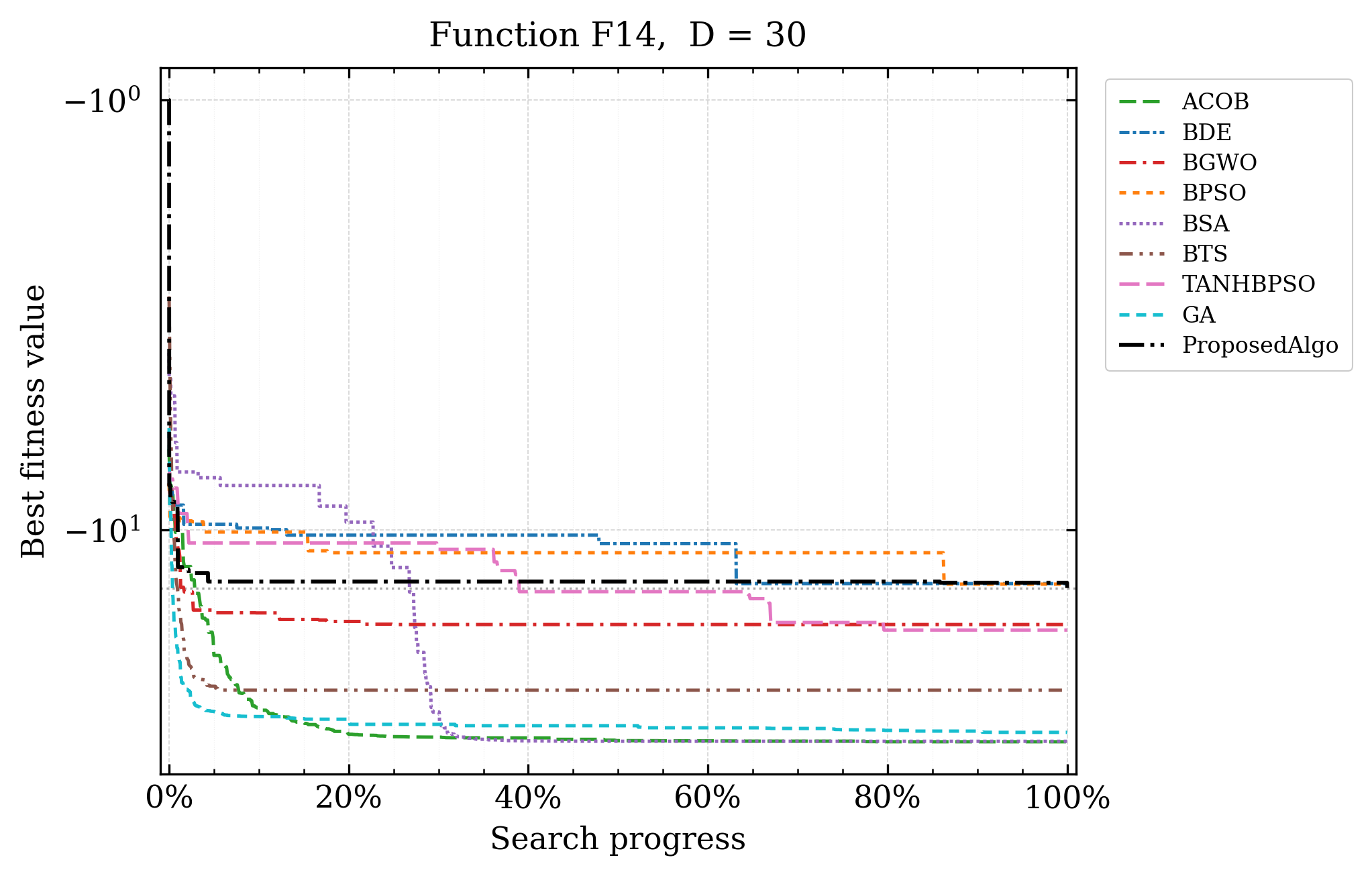} & 
\includegraphics[width=0.23\textwidth]{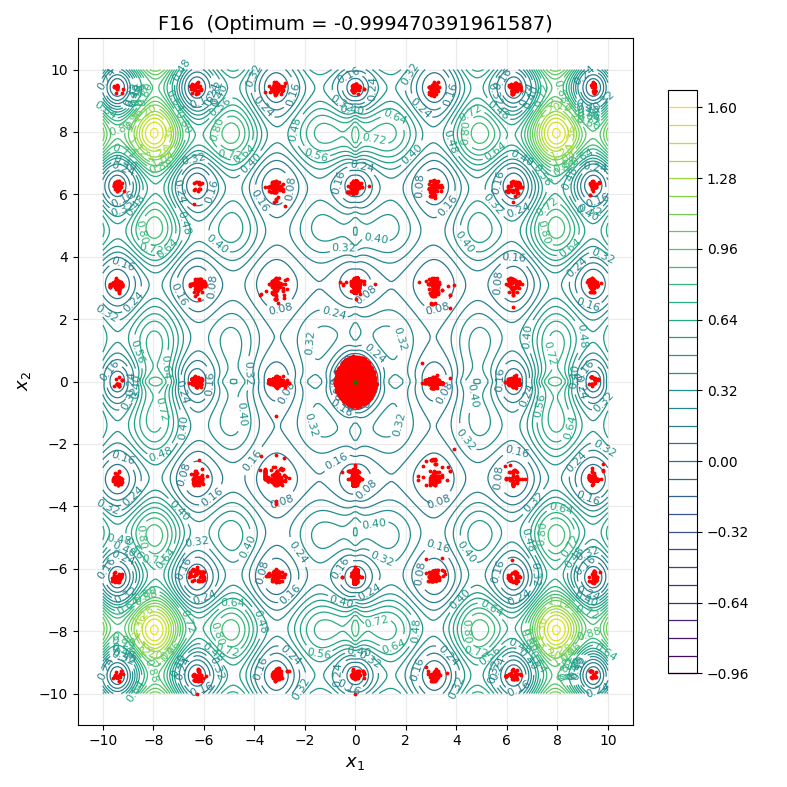} &
\includegraphics[width=0.25\textwidth]{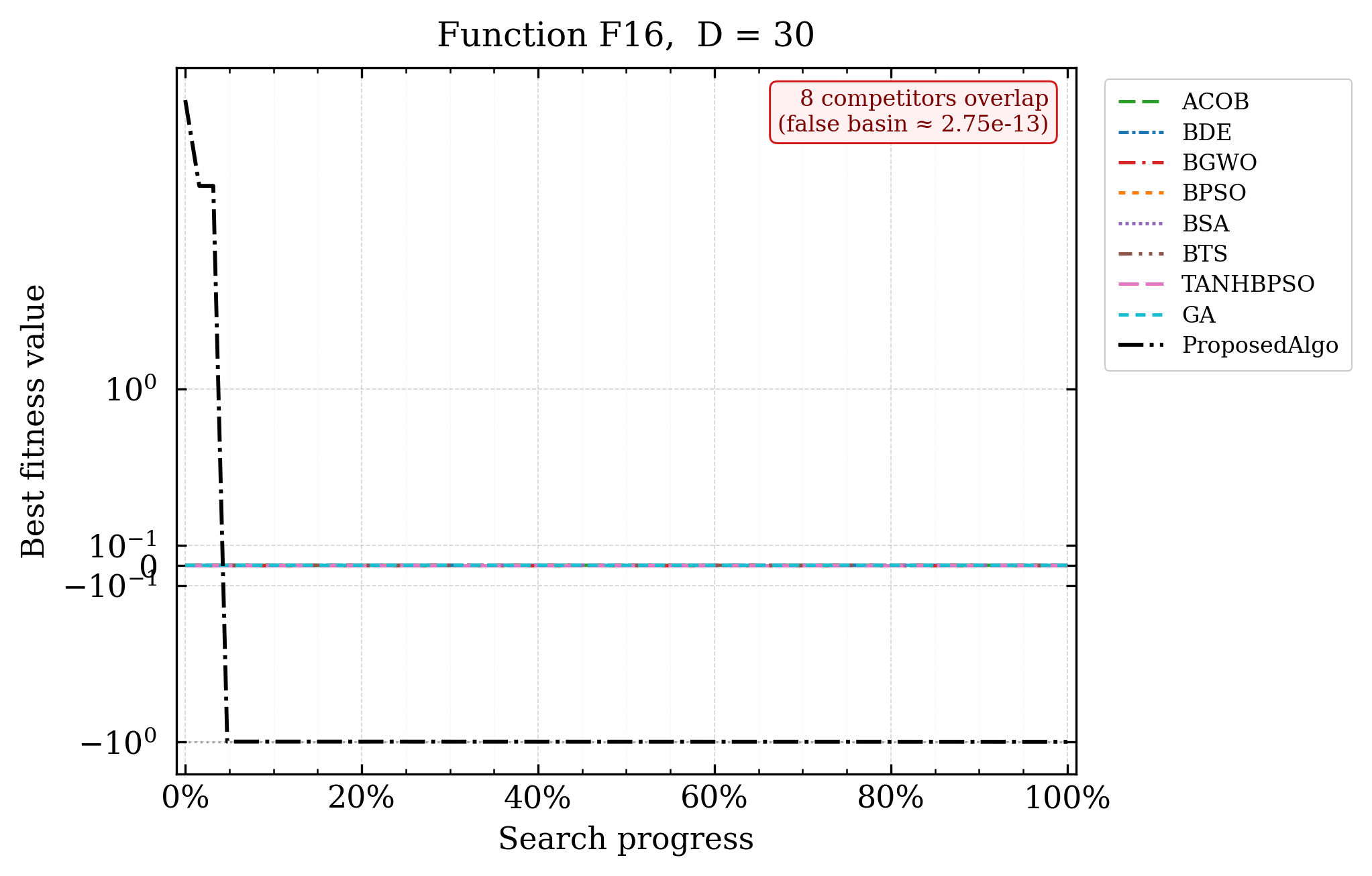} \\

\caption{Visual analysis of the S-LCG algorithm on selected scalable functions from (F1-F16). The columns display the exploitation history of the sampled points and the convergence curves tracking fitness evaluation over time.}
\label{fig:visual_unimodal}

\end{longtable}
\egroup

\textbf{Fixed-Dimension Landscapes (F17--F26):} are built to mimic real-world engineering problems and complicated mathematical structures. These functions are evaluated to determine the ability of the algorithm to function in complex, non-scalable search domains. On functions F17-F22, S-LCG is competitive with the stochastic baseline set by GA. But its structural dominance is displayed on the densely packed Shekel family (F24-F26). Stochastic algorithms spread their evaluations too widely because of the presence of deep, narrow, overlapping minima. S-LCG always finds the deepest basins, while the randomized evolutionary mechanisms are much worse. Table ~\ref{tab:results_mean_F17F26_combined} demonstrates the achieved result following the same statistical comparison, followed previously in the scalable table, to depict convergence and exploitation history, which has been demonstrated in figure ~\ref{fig:visual_fixed}.

\newcolumntype{C}{>{\centering\arraybackslash\scriptsize}p{1.5cm}}
\small
\setlength{\tabcolsep}{2pt}
\renewcommand{\arraystretch}{0.9}
\begin{longtable}{lcCCCCCCCCC}
\toprule
\textbf{Function} & \textbf{Dim} & \multicolumn{1}{c}{\textbf{S-LCG}} & \multicolumn{1}{c}{\textbf{GA}} & \multicolumn{1}{c}{\textbf{BDE}} & \multicolumn{1}{c}{\textbf{BPSO}} & \multicolumn{1}{c}{\textbf{BACO}} & \multicolumn{1}{c}{\textbf{T-BPSO}} & \multicolumn{1}{c}{\textbf{BTS}} & \multicolumn{1}{c}{\textbf{BSA}} & \multicolumn{1}{c}{\textbf{BGWO}} \\
\midrule
\endfirsthead

\bottomrule
\noalign{\vspace{6pt}} 
\caption{Mean Results on Benchmark Functions F17--F26 across dimensions $d \in \{2,3,4,6\}$}\label{tab:results_mean_F17F26_combined}\\
\endlastfoot

F17 & 2 & \phantom{-}9.98e-01\phantom{$^{+}$} & \phantom{-}1.03e+00$^{-}$ & \phantom{-}1.26e+00$^{+}$ & \phantom{-}1.05e+00$^{+}$ & \phantom{-}1.20e+00$^{+}$ & \phantom{-}1.04e+00$^{+}$ & \phantom{-}7.48e+00$^{+}$ & \phantom{-}5.87e+00$^{+}$ & \phantom{-}5.23e+00$^{+}$ \\
F19 & 2 & -1.03e+00\phantom{$^{+}$} & -1.03e+00$^{-}$ & -1.03e+00$^{+}$ & -1.03e+00$^{+}$ & -1.01e+00$^{+}$ & -1.03e+00$^{+}$ & -9.96e-01$^{+}$ & -1.00e+00$^{+}$ & -8.94e-01$^{+}$ \\
F20 & 2 & \phantom{-}3.98e-01\phantom{$^{+}$} & \phantom{-}3.98e-01$^{=}$ & \phantom{-}3.98e-01$^{+}$ & \phantom{-}3.98e-01$^{+}$ & \phantom{-}4.55e-01$^{+}$ & \phantom{-}4.00e-01$^{+}$ & \phantom{-}2.90e+00$^{+}$ & \phantom{-}1.08e+00$^{+}$ & \phantom{-}8.45e-01$^{+}$ \\
F21 & 2 & \phantom{-}3.00e+00\phantom{$^{+}$} & \phantom{-}3.00e+00$^{=}$ & \phantom{-}3.00e+00$^{+}$ & \phantom{-}3.00e+00$^{+}$ & \phantom{-}5.70e+00$^{-}$ & \phantom{-}3.00e+00$^{+}$ & \phantom{-}1.30e+02$^{+}$ & \phantom{-}6.30e+01$^{+}$ & \phantom{-}1.47e+01$^{+}$ \\

F22 & 3 & -3.86e+00\phantom{$^{+}$} & -3.86e+00$^{-}$ & -3.86e+00$^{+}$ & -3.86e+00$^{-}$ & -3.86e+00$^{=}$ & -3.86e+00$^{=}$ & -2.77e+00$^{+}$ & -3.85e+00$^{+}$ & -3.78e+00$^{+}$ \\

F18 & 4 & \phantom{-}3.67e-04\phantom{$^{+}$} & \phantom{-}2.51e-03$^{+}$ & \phantom{-}1.04e-03$^{+}$ & \phantom{-}6.59e-04$^{+}$ & \phantom{-}9.08e-03$^{+}$ & \phantom{-}1.07e-03$^{+}$ & \phantom{-}3.65e-02$^{+}$ & \phantom{-}4.05e-03$^{+}$ & \phantom{-}1.23e-02$^{+}$ \\
F24 & 4 & -1.02e+01\phantom{$^{+}$} & -6.98e+00$^{-}$ & -5.57e+00$^{+}$ & -6.59e+00$^{+}$ & -4.05e+00$^{+}$ & -6.75e+00$^{+}$ & -2.08e+00$^{+}$ & -4.17e+00$^{+}$ & -2.24e+00$^{+}$ \\
F25 & 4 & -1.04e+01\phantom{$^{+}$} & -6.82e+00$^{+}$ & -7.30e+00$^{+}$ & -8.97e+00$^{+}$ & -4.71e+00$^{+}$ & -7.26e+00$^{+}$ & -2.00e+00$^{+}$ & -4.56e+00$^{+}$ & -2.53e+00$^{+}$ \\
F26 & 4 & -1.05e+01\phantom{$^{+}$} & -7.04e+00$^{+}$ & -8.09e+00$^{+}$ & -8.52e+00$^{+}$ & -4.53e+00$^{+}$ & -8.05e+00$^{+}$ & -1.95e+00$^{+}$ & -4.52e+00$^{+}$ & -3.03e+00$^{+}$ \\

F23 & 6 & -3.11e+00\phantom{$^{+}$} & -3.24e+00$^{-}$ & -3.17e+00$^{-}$ & -3.24e+00$^{-}$ & -3.21e+00$^{-}$ & -3.22e+00$^{-}$ & -2.78e+00$^{+}$ & -3.23e+00$^{-}$ & -3.12e+00$^{=}$ \\
\midrule
\multicolumn{2}{l}{\textbf{+/=/-}} & \multicolumn{1}{c}{-} & \multicolumn{1}{c}{3/2/5} & \multicolumn{1}{c}{9/0/1} & \multicolumn{1}{c}{8/0/2} & \multicolumn{1}{c}{7/1/2} & \multicolumn{1}{c}{8/1/1} & \multicolumn{1}{c}{10/0/0} & \multicolumn{1}{c}{9/0/1} & \multicolumn{1}{c}{9/1/0} \\
\end{longtable}

\bgroup
\renewcommand\LTcaptype{figure} 

\begin{longtable}{cccc}

\textbf{Exploitation History} & \textbf{Convergence} & \textbf{Exploitation History} & \textbf{Convergence} \\[2ex]

\includegraphics[width=0.23\textwidth]{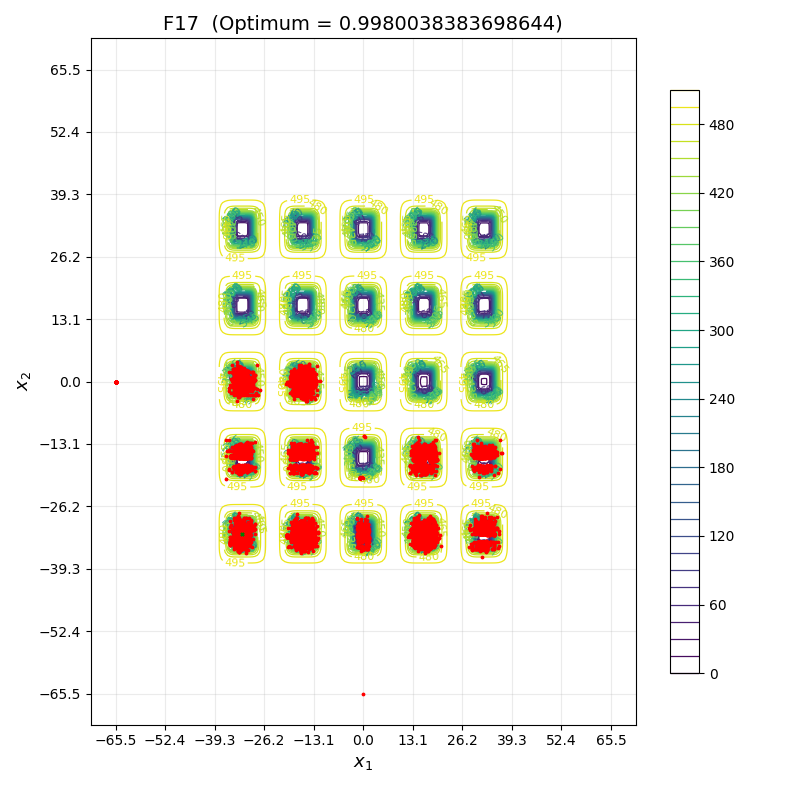} &
\includegraphics[width=0.25\textwidth]{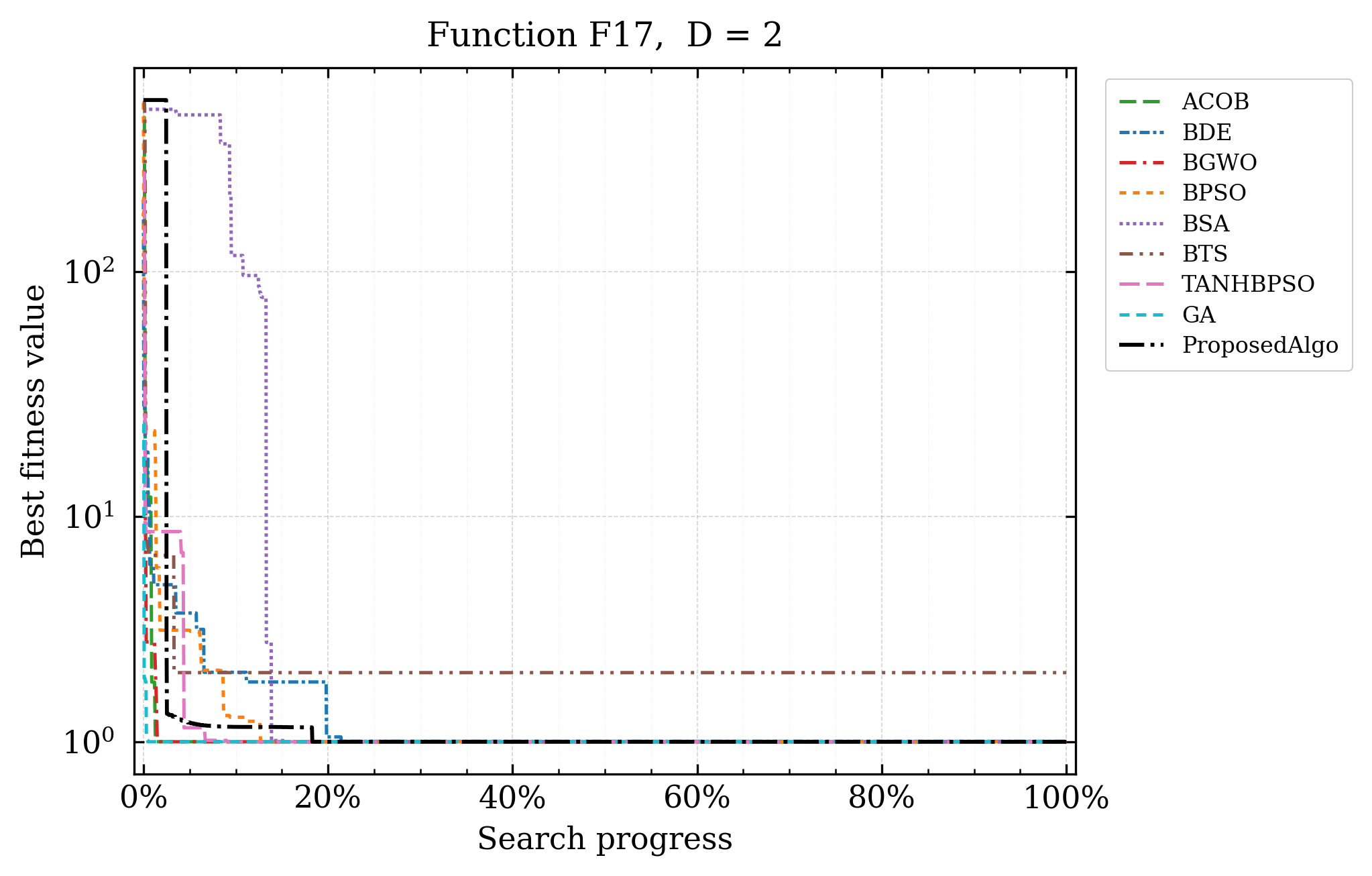} & 
\includegraphics[width=0.23\textwidth]{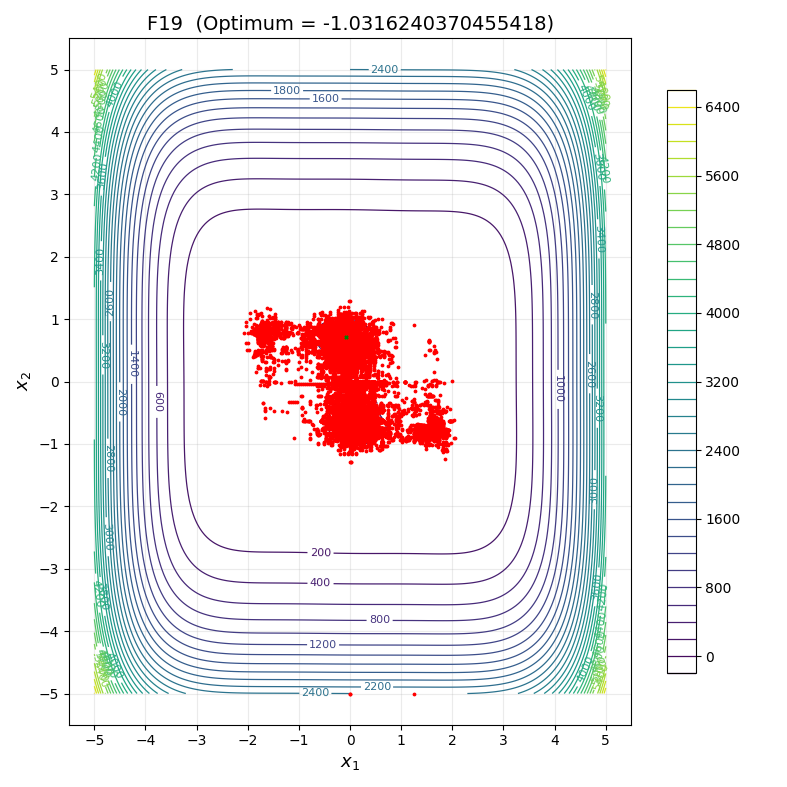} &
\includegraphics[width=0.25\textwidth]{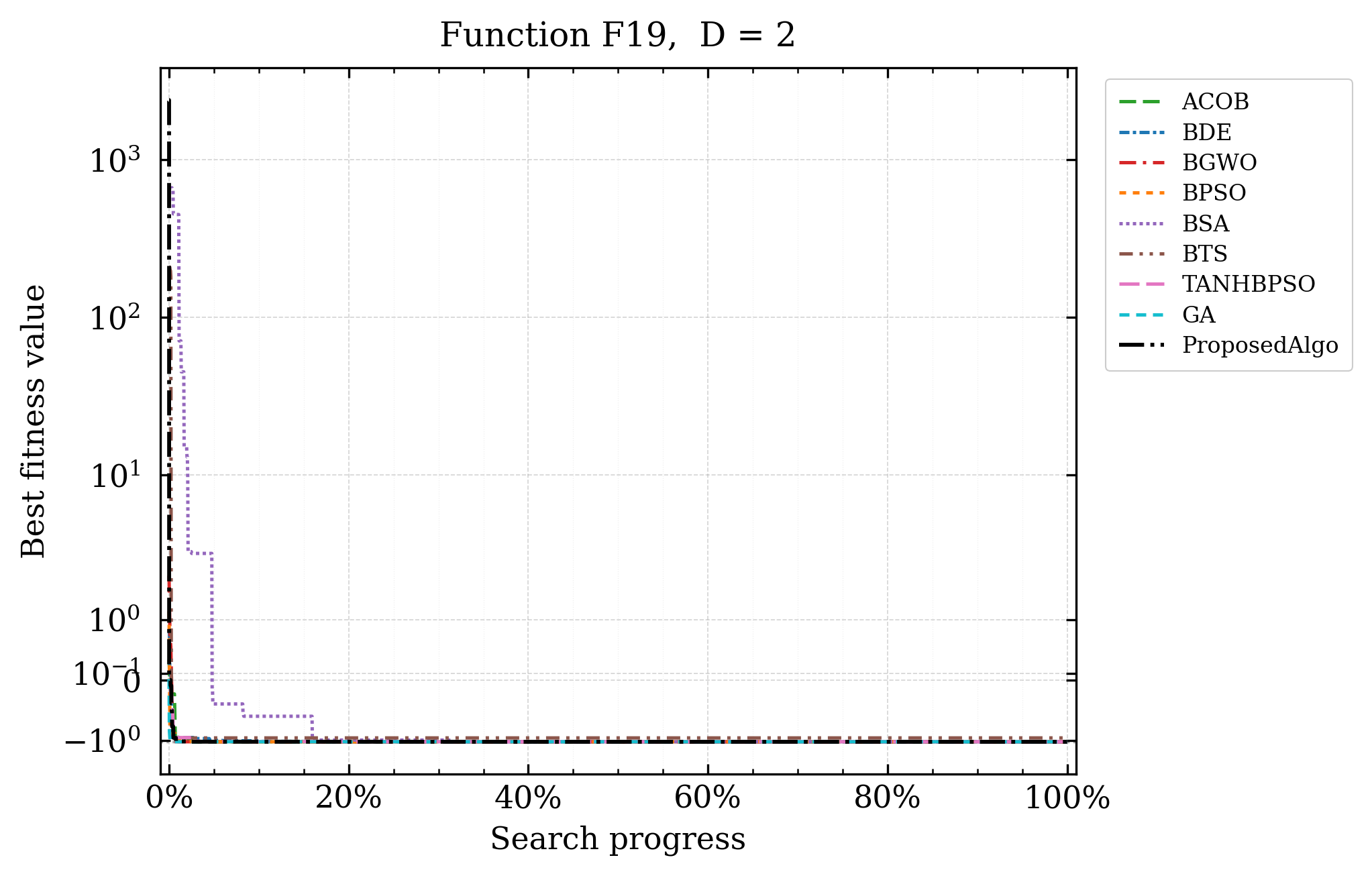} \\[2ex]

\includegraphics[width=0.23\textwidth]{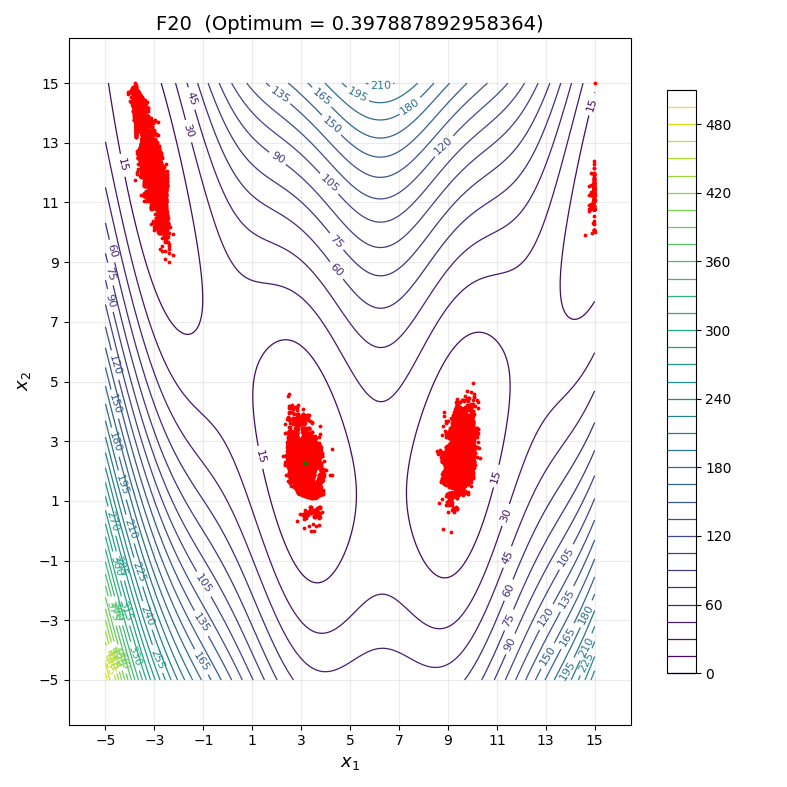} &
\includegraphics[width=0.25\textwidth]{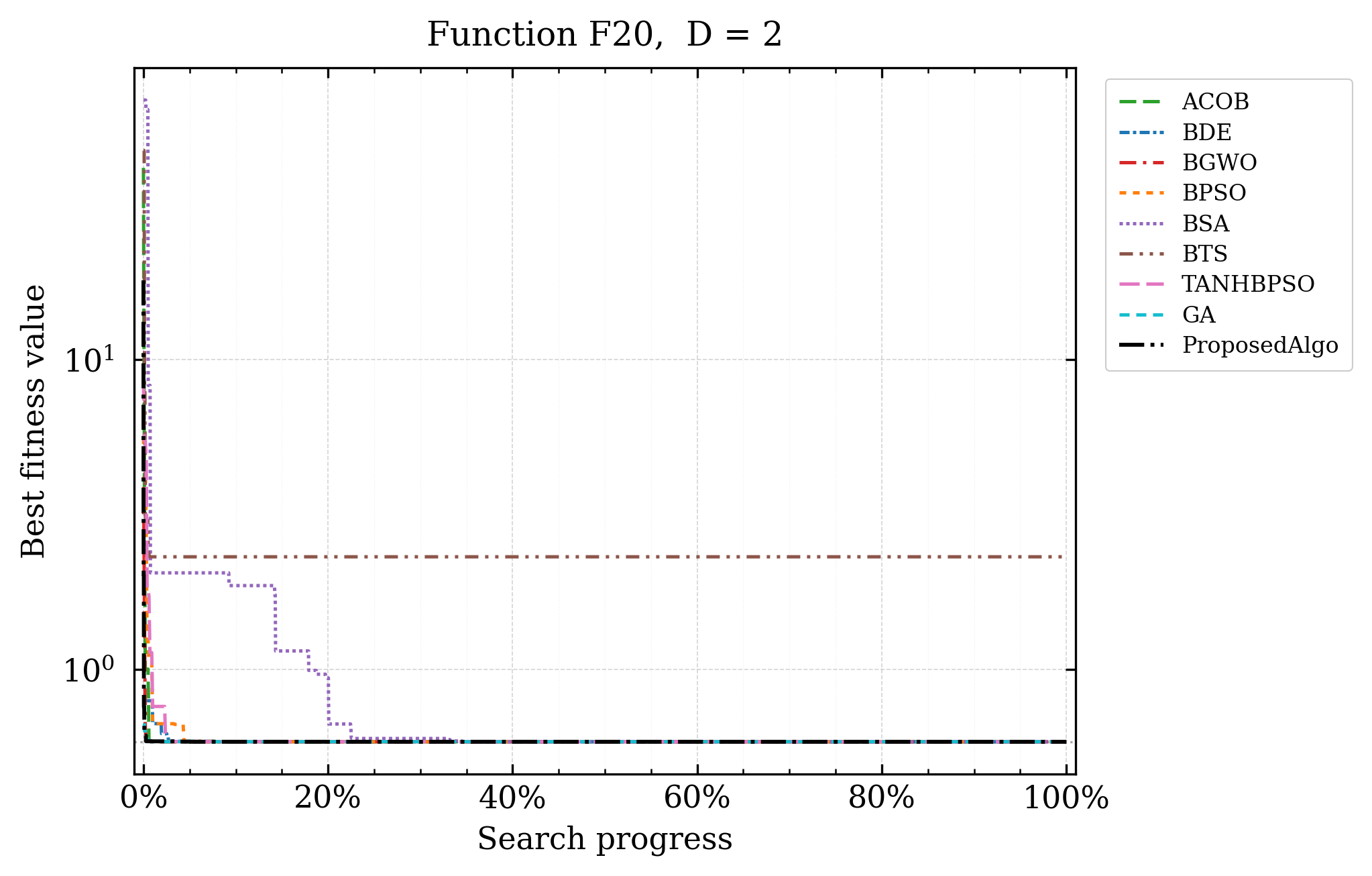} & 
\includegraphics[width=0.23\textwidth]{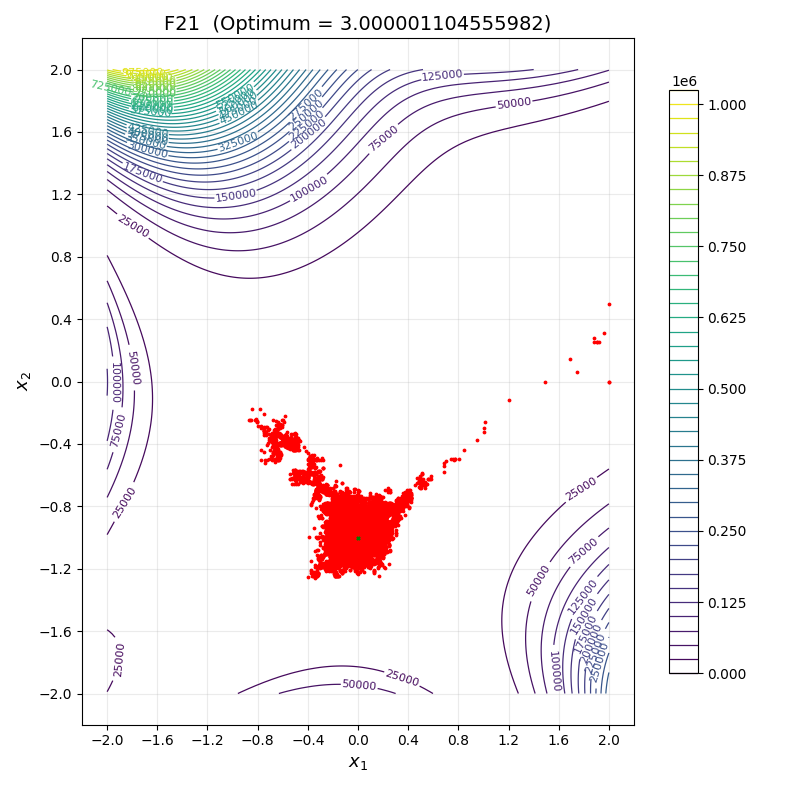} &
\includegraphics[width=0.25\textwidth]{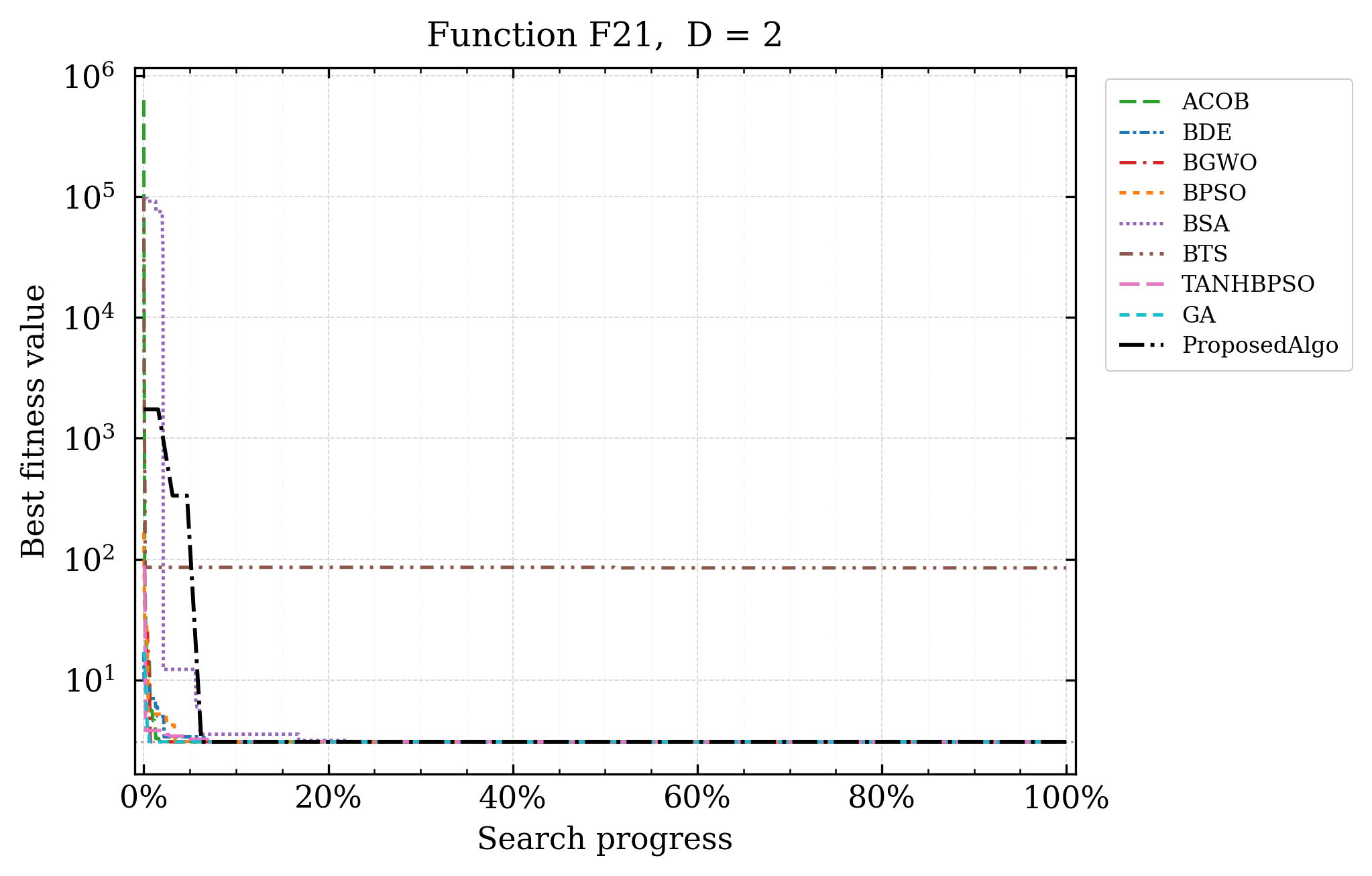} \\[2ex]

\caption{Visual analysis of the S-LCG algorithm on fixed-dimension multimodal functions.}
\label{fig:visual_fixed}

\end{longtable}
\egroup

\subsection{Constrained Engineering Design Problems}\label{sec:engineering}

To verify the practical usefulness of S-LCG, it was applied to three classical constrained engineering problems~\cite{GWO}: Tension/Compression Spring, Welded Beam, and Pressure Vessel design. Infeasible solutions were handled with a strict death penalty ($f = +\infty$).
Summing up, the results prove the S-LCG to be uniquely robust in the context of statistical methods. For the Spring and Welded Beam continuous designs the S-LCG was compared to or got less optimal costs than the competitors. On the other hand, the Pressure Vessel problem presents a mixed variable environment, as both discrete thickness levels and continuous dimensions are tested. Continuous meta-heuristics incur post-hoc rounding penalties that frequently violate constraints. The native bit-splitting encoding of S-LCG assigns exact discrete subsets to thickness variables. So, S-LCG converged to the minimum cost ($6145.06$) reliably and beat all 8 competitors. All results presented in the tables (~\ref{tab:constraint_problem1},~\ref{tab:constraint_problem2},~\ref{tab:constraint_problem3}) while convergebce plot depicted in figures (~\ref{fig:spring_convergence},~\ref{fig:welded_convergence},~\ref{fig:vessel_convergence}) 

\begin{figure}[H]
    \centering
    
    \begin{minipage}[t]{0.55\textwidth}
        \vspace{0pt} 

        \resizebox{\linewidth}{!}{
            \begin{tabular}{lcccccc}
            \toprule
            \multirow{2}{*}{\textbf{Algorithm Name}} & \multicolumn{3}{c}{\textbf{Dimension}} & \multirow{2}{*}{\textbf{Achieved Solution}} & \multirow{2}{*}{\textbf{Mean}} & \multirow{2}{*}{\textbf{STD}} \\ \cmidrule(lr){2-4}
             & \textbf{$x_1$} & \textbf{$x_2$} & \textbf{$x_3$} & & & \\ 
            \midrule
            S-LCG & 0.05435 & 0.42127 &  8.47099 & 1.303e-02 & - & - \\
            BPSO &  0.05279 & 0.38344 & 10.48835 & 1.268e-02 $+$ & 1.31e-02 & 2.96e-04 \\
            BDE &   0.05253 & 0.37499 & 10.41615 & 1.285e-02 $+$ & 1.37e-02 & 4.97e-04 \\
            GA &    0.05333 & 0.39768 &  9.23791 & 1.271e-02 $+$ & 1.67e-02 & 6.04e-03 \\
            T-PSO & 0.05179 & 0.35826 & 11.31447 & 1.280e-02 $+$ & 1.44e-02 & 1.10e-03 \\
            BSA & 0.05547 & 0.45388 & 7.30867 & 1.300e-02 $+$ & 2.38e-02 & 3.36e-02 \\
            BACO & 0.05261 & 0.37919 & 10.12500 & 1.273e-02 $+$ & 4.67e+14 & 5.07e+14 \\
            BGWO & 0.05761 & 0.51250 & 5.88476 & 1.341e-02 $+$ & 2.11e-02 & 8.99e-03 \\
            BTS & 0.05856 & 0.54531 & 5.250001 & 1.356e-02 $+$ & 6.52e-02 & 6.87e-02 \\
            \midrule
            $+/=/-$ & \multicolumn{3}{c}{} & 8/0/0 & & \\ 
            \bottomrule
            \end{tabular}
        }
        \captionof{table}{\footnotesize Achieved result for tension/compression spring problem}
        \label{tab:constraint_problem1}

    \end{minipage}\hfill 
    \begin{minipage}[t]{0.35\textwidth}
        \vspace{0pt} 
        \centering
        
        \includegraphics[width=\linewidth]{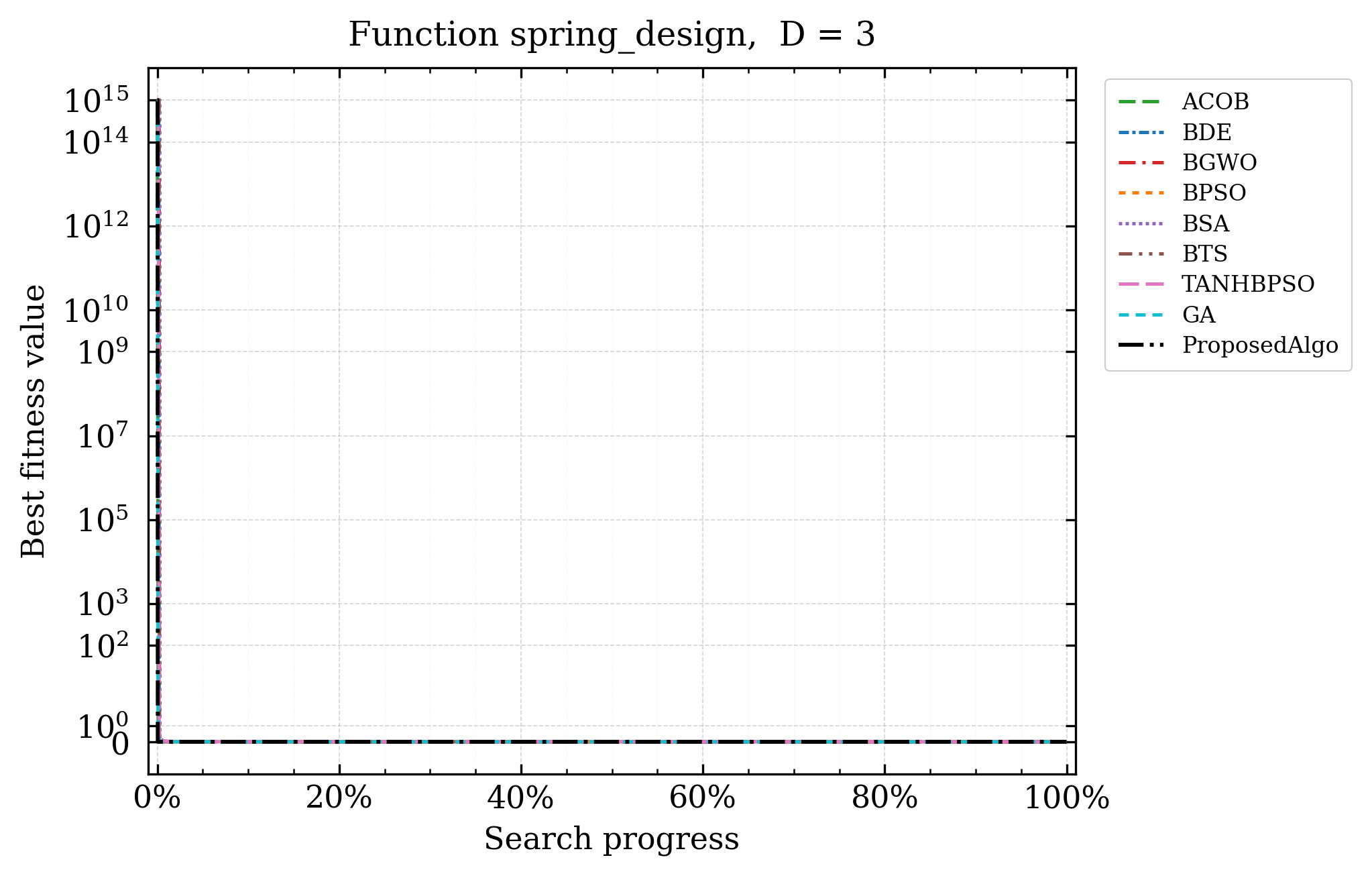}
        
        \captionof{figure}{Convergence curves for the tension/compression spring.}
        \label{fig:spring_convergence}
    \end{minipage}
\end{figure}
\begin{figure}[H]
    \centering
    
    \begin{minipage}[t]{0.55\textwidth}
        \vspace{0pt} 

        \resizebox{\linewidth}{!}{
            \begin{tabular}{lccccccc}
            \toprule
            \multirow{2}{*}{\textbf{Algorithm Name}} & \multicolumn{4}{c}{\textbf{Dimension}} & \multirow{2}{*}{\textbf{Achieved Solution}} & \multirow{2}{*}{\textbf{Mean}} & \multirow{2}{*}{\textbf{STD}} \\ \cmidrule(lr){2-5}
              & \textbf{$x_1$} & \textbf{$x_2$} & \textbf{$x_3$} & \textbf{$x_4$} & & & \\ 
            \midrule
S-LCG & 0.20568 & 3.47837 & 9.03680 & 0.20763 & 1.74036 & - & - \\
BPSO &0.20654&3.68103&9.03752&0.20729&1.76711 $+$ & 1.81693e+00 & 3.61520e-02\\
BDE &0.20191&3.64236&9.11515&0.20702&1.76574 $+$ & 1.84824e+00 & 3.78380e-02\\
T-BPSO &0.19176&3.93923&9.004198&0.20987 &1.79101 $+$& 1.90310e+00 & 7.89416e-02\\
GA &0.21875&3.30976&8.76341&0.21875&1.77144 $+$ & 2.14566e+00 & 3.58402e-01\\
BACO &0.21689&3.34844&8.76371&0.21875 &1.77406 $+$& 2.44607e+00 & 7.09917e-01\\
BGWO &0.21875&3.50312&8.76250&0.21880&1.79970 $+$& 2.92176e+00 & 6.82489e-01\\
BSA &0.21874&3.30978&8.76371&0.21875&1.77144 $+$& 7.33333e+06 & 4.49776e+06\\
BTS &0.34202&2.19764&7.53467&0.34202&2.29220 $+$& 6.00000e+06 & 4.98273e+06\\
\midrule
$+/=/-$ & \multicolumn{4}{c}{} & 8/0/0 & & \\ 
\bottomrule
\end{tabular}
        }
        \captionof{table}{\footnotesize Achieved result for welded beam design problem.}
        \label{tab:constraint_problem2}

    \end{minipage}\hfill 
    \begin{minipage}[t]{0.35\textwidth}
        \vspace{0pt} 
        \centering
        \includegraphics[width=\linewidth]{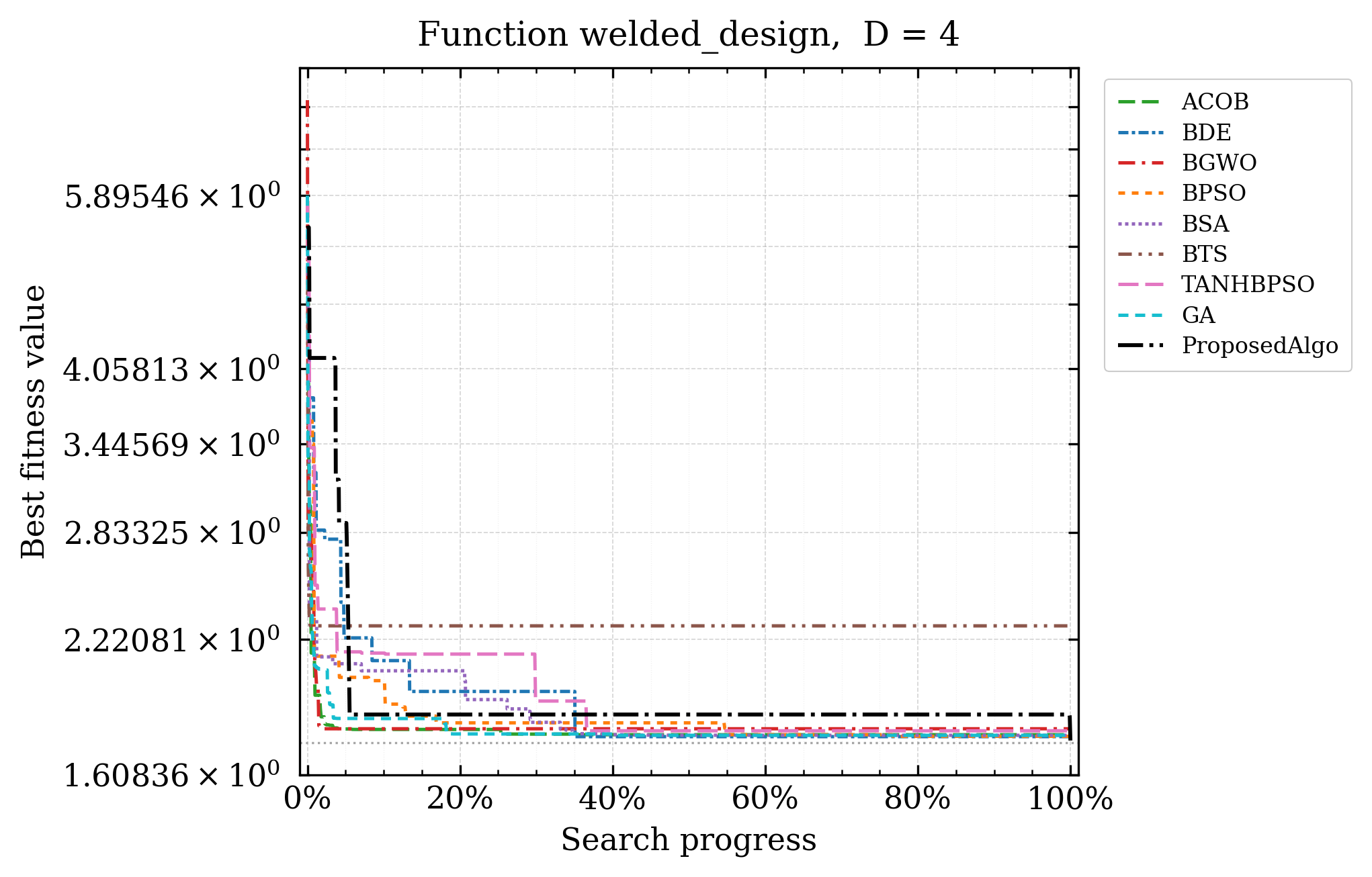}      
        \captionof{figure}{Convergence curves for the welded beam design problem.}
        \label{fig:welded_convergence}
    \end{minipage}
\end{figure}

\begin{figure}[H]
    \centering
    
    \begin{minipage}[t]{0.55\textwidth}
        \vspace{0pt} 

        \resizebox{\linewidth}{!}{
\begin{tabular}{lccccccc}
\toprule
\multirow{2}{*}{\textbf{Algorithm Name}} & \multicolumn{4}{c}{\textbf{Dimension}} & \multirow{2}{*}{\textbf{Achieved Cost}} & \multirow{2}{*}{\textbf{Mean}} & \multirow{2}{*}{\textbf{STD}} \\ \cmidrule(lr){2-5}
 & $\textbf{$x_1$}$ & $\textbf{$x_2$}$ & \textbf{$x_3$} & \textbf{$x_4$} & & & \\ 
\midrule
S-LCG & 0.8125 & 0.4375 &  41.45648 & 185.00077 & 6145.06920 & - & - \\

BDE &  0.8125 & 0.4375 & 42.08664 & 186.12366 & 6279.47710 $+$ & 6.463972e+03 & 1.43600e+02  \\

BPSO &   0.875 & 0.4375 & 44.84471 & 146.39341 & 6175.77147 $+$ & 6.758928e+03 & 3.83635e+02  \\

GA &    0.875 & 0.4375 & 44.17487 & 152.50006 & 6227.52031 $+$ & 6.83492e+03 & 3.973271e+02  \\

T-PSO & 0.875 & 0.4375 & 43.95289 & 155.85205  & 6278.85570 $+$ & 6.93534e+03&3.98748e+02 \\

BACO & 0.875 & 0.4375 & 44.17541  & 152.50006 & 6227.61108 $+$ &7.350988e+03 & 8.04005e+02  \\

BGWO & 0.875& 0.5625&44.18701&152.50006 & 6663.51329 $+$ & 2.71939e+03 & 6.66351e+03  \\

BTS & 1.375 & 0.625 & 63.71111 & 16.68258 & 7910.20035 $+$ &  3.98748e+02 & 6.27885e+03 \\

BSA & 1.0625 & 0.5625 & 54.27574 & 68.148678 & 6851.63959 $+$ &  3.48195e+04 &6.85164e+03  \\

\midrule
$+/=/-$ & \multicolumn{4}{c}{} & 8/0/0 & & \\ 
\bottomrule
\end{tabular}
        }
        \captionof{table}{\footnotesize Achieved result for the pressure vessel design problem.}
        \label{tab:constraint_problem3}

    \end{minipage}\hfill 
    \begin{minipage}[t]{0.35\textwidth}
        \vspace{0pt} 
        \centering
        \includegraphics[width=\linewidth]{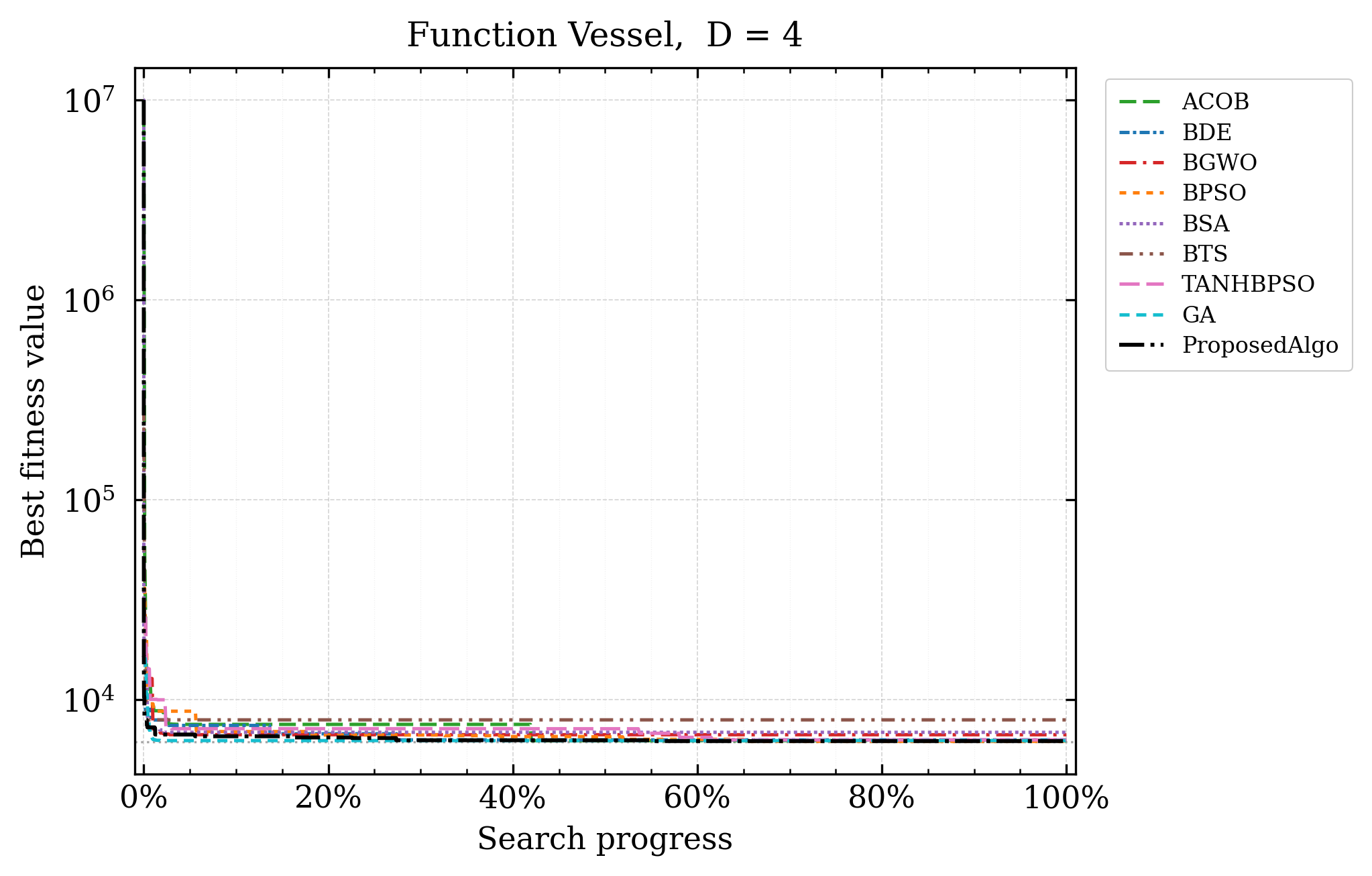}      
        \captionof{figure}{Convergence curves for the pressure vessel design problem.}
        \label{fig:vessel_convergence}
    \end{minipage}
\end{figure}

\subsection{Statistical and Efficiency Analysis}\label{sec:stats_efficiency}

Statistical validation is conducted using the Friedman mean rank test and the one-sample Wilcoxon signed-rank test because S-LCG yielded a single solution due to its deterministic design, compared with the median across 30 runs of the opposing algorithms with ($\alpha = 0.05$). Finally, to check the significance of the proposed algorithm, the Holm post hoc test was used.

\textbf{Friedman Rankings:} S-LCG and GA are the only algorithms to consistently occupy the top two tiers. Crucially, S-LCG's mean rank improves progressively with dimensionality: from $1.9$ at $d=2$ to an optimal $1.5$ at $d=30$. GA remains stagnant, while the remaining binary swarm and evolutionary methods degrade sharply.  Bold signifies the optimal (lowest) rank for each dimension. S-LCG and GA are the sole algorithms that consistently secure the top two places. as it is elaborated in the following table~\ref{tab:friedman_all}.


\begin{table}[H]
\small
\centering
\footnotesize
\setlength{\tabcolsep}{4pt}
\renewcommand{\arraystretch}{0.9}
\begin{tabular}{@{}lrrrrrrrrr@{}}
\toprule
Algorithm & $d=2$ & $d=3$ & $d=4$ & $d=5$ & $d=6$ & $d=10$ & $d=20$ & $d=30$ & all D \\
\midrule
S-LCG & \textbf{1.90} & \textbf{2.29} & \textbf{1.85} & \textbf{2.16} & 2.38 & \textbf{2.06} & \textbf{1.78} & \textbf{1.50} & \textbf{1.99} \\
GA & 2.35 & 2.88 & 2.88 & 2.16 & \textbf{2.03} & 2.44 & 2.72 & 2.81 & 2.54 \\
BACO & 4.92 & 4.62 & 4.67 & 3.72 & 3.09 & 2.94 & 2.66 & 3.00 & 3.77 \\
BSA & 7.03 & 5.26 & 5.20 & 4.72 & 4.26 & 3.69 & 3.91 & 3.69 & 4.80 \\
BPSO & 4.75 & 3.29 & 3.65 & 4.50 & 5.06 & 6.19 & 7.38 & 7.56 & 5.22 \\
T-BPSO & 5.20 & 6.18 & 5.95 & 7.00 & 6.59 & 6.56 & 6.19 & 6.25 & 6.20 \\
BTS & 6.90 & 7.18 & 7.50 & 6.38 & 6.47 & 5.19 & 4.94 & 4.81 & 6.24 \\
BDE & 3.70 & 4.82 & 5.00 & 5.94 & 7.00 & 7.94 & 8.44 & 8.69 & 6.31 \\
BGWO & 8.25 & 8.47 & 8.30 & 8.44 & 8.12 & 8.00 & 7.00 & 6.69 & 7.93 \\

\bottomrule
\end{tabular}
\caption{Friedman mean rankings for all dimensions.}
\label{tab:friedman_all}
\end{table}

\textbf{The One-Sample Wilcoxon W/T/L:} Across all dimensions used to check the performance of the proposed algorithm on each evaluated function against each competitor used, where S-LCG achieves a 52\% win rate against GA (the strongest competitor) and overwhelming dominance against the rest: a 91\% win rate against BDE, 93\% against T-BPSO, and 95\% against BGWO. Table~\ref{tab:wtl_all} demonstrates overall performance, while figure~\ref{fig:wtl} shows the comparison with reference to dimension.

\begin{figure}[H]
    \centering
    
    \begin{minipage}[H]{0.40\textwidth}
        \vspace{0pt} 
        
        \footnotesize
        \centering
        \setlength{\tabcolsep}{4pt}
        \renewcommand{\arraystretch}{0.9}
        
        \begin{tabular}{@{}lcccc@{}}
        \toprule
        vs.\ Algorithm & Total $+$ & Total $=$ & Total $-$ & Win rate \\
        \midrule
        GA         & 72      & 39      & 27      & 52\% \\
        ACOB       & 94      & 28      & 16      & 68\% \\
        BSA        & 96      & 25      & 17      & 70\% \\
        BTS        & 116     & 11      & 11      & 84\% \\
        BDE        & 126     & 3       & 9       & 91\% \\
        BPSO       & 129     & 0       & 9       & 93\% \\
        T-BPSO     & 128     & 3       & 7       & 93\% \\
        BGWO       & 131     & 2       & 5       & 95\% \\
        \bottomrule
        \end{tabular}
        
        \captionof{table}{Cumulative $+/=/-$ record of S-LCG versus each competitor over all combined dimensions}
        \label{tab:wtl_all}
        
    \end{minipage}\hfill 
    \begin{minipage}[ht]{0.50\textwidth}
        \vspace{0pt} 
        \centering
        \includegraphics[width=\linewidth,height=0.20\textheight]{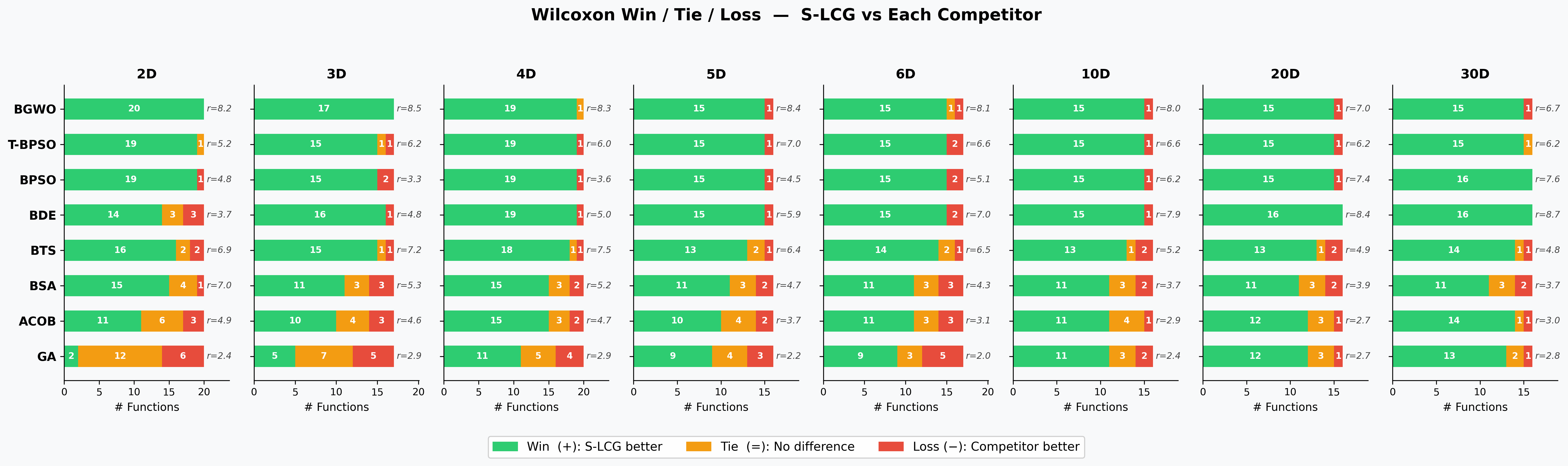}      
        
        \captionof{figure}{WTL comparison across dimensions.}
        \label{fig:wtl}
    \end{minipage}
\end{figure}

\textbf{The Holm Post Hoc Test:} used to evaluate the actual significance of the proposed algorithm against the competitor, where the Critical Difference (CD) Figure ~\ref{fig:CD} shows that the nearest competitor is GA, and table ~\ref{tab:holm_posthoc} it's show the calculation of Holm for 138 evaluated functions. 
\begin{figure}[H]
    \centering
    \includegraphics[width=0.7\linewidth]{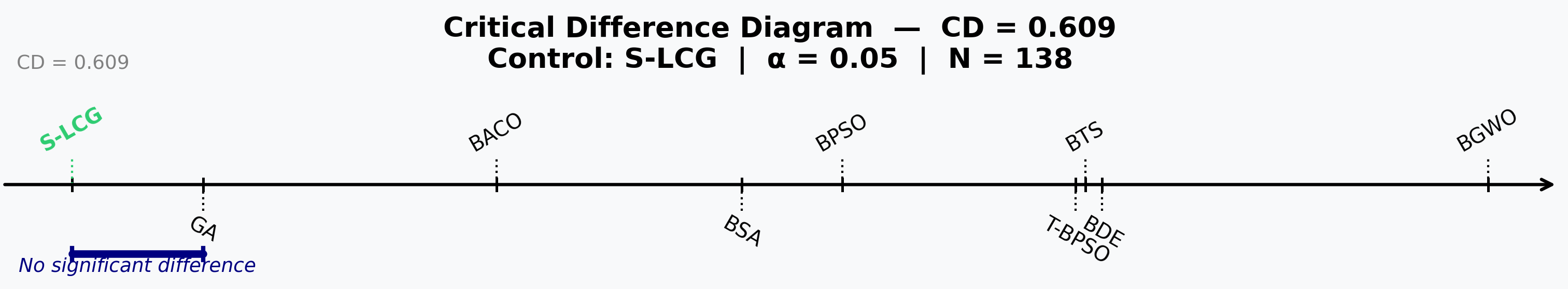}
    \caption{Critical Difference Diagram.}
    \label{fig:CD}
\end{figure}

\begin{table}[H]
\centering

\begin{tabular}{lcccccc}
\hline
\textbf{Competitor} & \textbf{Rank S-LCG} & \textbf{Rank Comp.} & \textbf{Z} & \textbf{$P_{raw}$} & \textbf{$P_{holm}$} & \textbf{Decision} \\
\hline
BSA     & 1.99 & 4.80 & 9.0402  & $< 0.0001$ & $< 0.0001$ & Significant \\
BPSO    & 1.99 & 5.22 & 10.3914 & $< 0.0001$ & $< 0.0001$ & Significant \\
T-BPSO  & 1.99 & 6.20 & 13.5442 & $< 0.0001$ & $< 0.0001$ & Significant \\
BTS     & 1.99 & 6.24 & 13.6729 & $< 0.0001$ & $< 0.0001$ & Significant \\
BDE     & 1.99 & 6.31 & 13.8981 & $< 0.0001$ & $< 0.0001$ & Significant \\
BGWO    & 1.99 & 7.93 & 19.1098 & $< 0.0001$ & $< 0.0001$ & Significant \\
BACO    & 1.99 & 3.77 & 5.7265  & $< 0.0001$ & $< 0.0001$ & Significant \\
GA      & 1.99 & 2.54 & 1.7694  & 0.0768     & 0.0768     & Not significant \\
\hline
\end{tabular}
\caption{Holm's Post-hoc Test Results -- Control: S-LCG ($\alpha = 0.05$, $N = 138$ functions, $k = 9$ algorithms)}
\label{tab:holm_posthoc}
\end{table}

\textbf{Computational Efficiency:} 
The increasing superiority of S-LCG in higher dimensions directly validates the mathematical propositions from Section~\ref{sec:theory}. As spatial volume expands exponentially, population-based methods suffer from convergence-induced redundancy, squandering evaluations. By maintaining an algebraically guaranteed constant information acquisition rate, S-LCG incurs a dimension-independent computational cost (evaluating roughly 35,000 to 50,000 generators across all problems, as detailed in the Supplementary Material). Furthermore, S-LCG circumvents the overhead of population sorting, velocity updates, and PRNG calls, operating with an absolute $O(d)$ memory complexity.

\section{Conclusions and Future Directions}\label{sec:conclusion}

    The S-LCG algorithm is introduced as a deterministic meta-heuristic optimisation algorithm. This means that S-LCG can be fully reproduced without statistical averaging. The method achieves search diversification by using algebraic structures, such as disjoint orbits and modular arithmetic, rather than stochastic operators. It employs a two-level nested loop architecture, including an adaptive outer loop for seed navigation and an exhaustive inner loop for $S$-LCG cycle assessment, ensuring that function evaluation is never redundant and that information is always gathered at a constant rate. Moreover, S-LCG requires only one sensitive parameter ($\Delta_{\text{step}}$), unlike other population-based meta-hueuristic methods, which require more than one. The algorithm also implicitly optimises a structurally smoother surrogate function $g(\alpha)$ and uses bit-splitting encoding to mitigate Marsaglia's effect, making it very efficient for expensive functions.

    The algorithm was evaluated against eight binary optimization methods (GA, ACOB, BDE, BGWO, BPSO, BSA, BTS, T-BPSO) and over 26 benchmark functions with dimensions ranging from $d = 2$ to $d = 30$. In the majority of instances, the S-LCG achieved outcomes within $1\%$ of the established continuous optimum. On the deceptive Schwefel function (F8), S-LCG performed particularly well, reaching within $0.01\%$ of the global optimum across all dimensions. This was achievable due to the algorithm's guaranteed zero-redundancy coverage, which ensures sampling within the basin of attraction of the isolated global optimum despite the terrain's difficulty. On Rastrigin (F9), which has around $10^d$ local minima, S-LCG achieved results almost as good as those of a machine even at $d = 30$. The Friedman ranking analysis revealed distinct performance, ranking highest across all tested dimensions. This substantiates the notion that the zero-redundancy quality becomes increasingly advantageous as the search space expands exponentially.

    Future research will focus primarily on five principal domains, namely formulating a formal convergence theory by connecting the coverage fraction $\rho(K)$ with the probability of attaining the global optimum's basin of attraction, executing adaptive $\Delta$ schemes that modify step sizes according to the local autocorrelation of the surrogate function $g(\alpha)$, utilizing the disjoint orbit structure for communication-free parallelization, facilitating linear scalability across $P$ workers, and broadening the framework to encompass multi-objective and constrained optimization, while integrating S-LCG's zero-redundancy global sweep with local refinement techniques such as Nelder-Mead or L-BFGS-B.

\newpage


    \bibliographystyle{unsrt}  
\bibliography{references}  

\end{document}